\documentclass[runningheads]{llncs}
\usepackage{bbding}
\usepackage[T1]{fontenc}
\usepackage{amsmath}
\usepackage{amsfonts}
\usepackage{xcolor}
\usepackage{amssymb}
\usepackage{algorithm}% 
\usepackage[noend]{algpseudocode}
\usepackage[utf8]{inputenc} % allow utf-8 input
\usepackage[T1]{fontenc}    % use 8-bit T1 fonts
\usepackage{lipsum}
\usepackage{hyperref}       % hyperlinks
\usepackage{url}            % simple URL typesetting
\usepackage[capitalize]{cleveref}

\usepackage{booktabs}       % professional-quality tables
\usepackage{amsfonts}       % blackboard math symbols

\usepackage{nicefrac}       % compact symbols for 1/2, etc.
\usepackage{microtype}      % microtypography
\usepackage{graphicx}
\usepackage[numbers]{natbib}
\usepackage{doi}
\usepackage{graphicx}
\usepackage{array}
\usepackage{amssymb}
\usepackage{enumerate}
\usepackage{enumitem}
\usepackage{bbm}
\usepackage{float}
\usepackage{mathtools}
\usepackage{setspace}
\usepackage{natbib}
\mathtoolsset{showonlyrefs = false}

\usepackage{doi}
\usepackage{empheq}
\usepackage{caption}
\usepackage{algorithm}% http://ctan.org/pkg/algorithms
\usepackage[noend]{algpseudocode}

\usepackage{subcaption}
\usepackage{ifsym}
\ifpdf
  \DeclareGraphicsExtensions{.eps,.pdf,.png,.jpg}
\else
  \DeclareGraphicsExtensions{.eps}
\fi

\usepackage{enumitem}
\setlist[enumerate]{leftmargin=.5in}
\setlist[itemize]{leftmargin=.5in}

\usepackage{subcaption}

\usepackage{amsthm}

\newcommand{\E}{\mathbb{E}}
\newtheorem{assumption}{Assumption}
\crefname{assumption}{assumption}{assumptions}
\Crefname{assumption}{Assumption}{Assumptions}

\newtheorem*{notation*}{Notation}

\usepackage{graphicx}
\newcommand{\rv}[2][final]{%
  \ifthenelse{\equal{#1}{revision}}%
    {\textcolor{blue}{#2}}%  % If 'revision', color the text red
    {#2}%                    % If 'final', keep normal text
}

\DeclareUnicodeCharacter{2212}{-}
\begin{document}
\title{Bilevel Learning with Inexact Stochastic Gradients}
%
% \titlerunning{Stochastic Bilevel Learning}
% If the paper title is too long for the running head, you can set
% an abbreviated paper title here
%
% \author{Mohammad Sadegh Salehi\inst{1}\and Subhadip Mukherjee\inst{2}\and Lindon Roberts\inst{3}\and Matthias J. Ehrhardt \inst{1}}

\author{Mohammad Sadegh Salehi\inst{1}\textsuperscript{(\Envelope)}\orcidID{0009−0002−1814-5024} \and
Subhadip Mukherjee\inst{2}\orcidID{0000-0002-7957-8758} \and
Lindon Roberts\inst{3}\orcidID{0000-0001-6438-9703} \and Matthias J. Ehrhardt \inst{1}\orcidID{0000-0001-8523-353X}}
\authorrunning{M. S. Salehi et al.}
% First names are abbreviated in the running head.
% If there are more than two authors, 'et al.' is used.
%
\institute{Department of Mathematical Sciences, University of Bath, Bath, BA2 7AY, UK \and
Department of Electronics \& Electrical Communication Engineering, Indian Institute of Technology (IIT) Kharagpur, India\\
\and
School of Mathematics and Statistics, University of Sydney, Camperdown NSW 2006, Australia\\
}
% \institute{}
%
\maketitle              % typeset the header of the contribution
\begin{abstract}
% The abstract should briefly summarize the contents of the paper in
% 150--250 words.
Bilevel learning has gained prominence in machine learning, inverse problems, and imaging applications, including hyperparameter optimization, learning data-adaptive regularizers, and optimizing forward operators. The large-scale nature of these problems has led to the development of inexact and computationally efficient methods. Existing adaptive methods predominantly rely on deterministic formulations, while stochastic approaches often adopt a doubly-stochastic framework with impractical variance assumptions, enforces a fixed number of lower-level iterations, and requires extensive tuning.
%limited discussion on the impact of inexactness in lower-level evaluations \LR{is the issue `limited discussion', or that they require significant tuning?}.
In this work, we focus on bilevel learning with strongly convex lower-level problems and a nonconvex sum-of-functions in the upper-level. Stochasticity arises from data sampling in the upper-level %, while deterministic methods are employed as the lower-level solver, 
which leads to inexact stochastic hypergradients.
%\LR{be more specific: strongly convex lower level problems with nonconvex sum-of-functions upper level, and consider deterministic methods at lower level with SGD in upper level, leading to inexact SGD (i.e.~abstract matches the title)} in the upper-level problem and employ approximate hypergradients.
We establish their connection to state-of-the-art stochastic optimization theory for nonconvex objectives. %, particularly practical assumptions on variance. 
Furthermore, we prove the convergence of inexact stochastic bilevel optimization under mild assumptions. %We demonstrate its practical benefits and 
Our empirical results highlight significant speed-ups and improved generalization in imaging tasks such as image denoising and deblurring in comparison with adaptive deterministic bilevel methods.      
%161 words

\end{abstract}

\section{Introduction}
Bilevel learning has emerged as a powerful framework in various machine learning domains such as hyperparameter optimization \cite{Pedregosa1602, grazzi2021convergence} and meta-learning \cite{Meta_learning-franceschi18a}. In inverse problems and imaging applications, bilevel optimization plays a key role within variational regularization frameworks \cite{Holler_2018, Kunisch2013ABO, Reyes2023}. In this work, following the convention from \cite{SGD_better}, considering a sampling vector \( v \in \mathbb{R}^m \) such that \( v \sim \mathcal{D} \) and \( \mathbb{E}_\mathcal{D}[v_i] = 1 \) for all $i\in \{1,\dots,m \}$, where \( \mathcal{D} \) is a user-defined distribution (for example, the uniform distribution \(\mathcal{D} = U(0,m)\)), we focus on the bilevel optimization problem described as follows:
% \begin{subequations}\label{GENERAL}
%     \begin{align}
%         \min_{\theta \in \mathbb{R}^d} \{f(\theta) &\coloneqq \frac{1}{m}\sum_{i=1}^{m} f_i(\theta) = \frac{1}{m}\sum_{i=1}^{m}g_i(\hat{x}_i(\theta))\}\label{UPPER_LEVEL}\\ s.t. \quad
%          {\hat{x}_i(\theta)} &\coloneqq \arg\min_{x\in \mathbb{R}^n} h_i(x,\theta), \quad i = 1,2,\dots, m.
%          \label{LOWER_LEVEL}
%     \end{align}
% \end{subequations}
\begin{subequations}\label{upper_stochastic}
\begin{align}
    \min_\theta \mathbb{E}_{v \sim \mathcal{D}}\Big[f_v(\theta) &\coloneqq \frac{1}{m} \sum_{i=1}^{m} v_i f_i(\theta) = \frac{1}{m}\sum_{i=1}^{m}v_ig_i(\hat{x}_i(\theta)) \Big], \label{Upper_stochastic}\\
 s.t. \quad
         {\hat{x}_i(\theta)} &\coloneqq \arg\min_{x\in \mathbb{R}^n} h_i(x,\theta), \quad i = 1,2,\dots, m,  \label{lower_stochastic}  
\end{align}
\end{subequations}
Here, the upper-level problem is a stochastic optimization problem where the upper-level functions  $g_i$  measure the quality of the solution to the lower-level problem \eqref{lower_stochastic}, for example, $g_i(x)= \|x - x^*_i\|^2$, where $x^*_i$ is the desired solution of the lower-level problem. In variational regularization, the lower-level objective is often formulated as $h_i(x,\theta) = D(Ax,y_i) + R_\theta(x),$ for all $i = 1,2,\dots, m,$
% \begin{equation}\label{lower_variational}
%     \hat{x}_i(\theta) = \arg\min_{x\in \mathbb{R}^n} \big \{h_i(x,\theta) = D(Ax,y) + R_\theta(x) \big \}, \quad i = 1,2,\dots, m,
% \end{equation}
where $D$ denotes the data fidelity term, each $y_i$ represents the noisy measurements, $A$  is the forward operator (often linear) for the imaging problems, and $R_\theta$ is a regularizer parametrized by $\theta$. Since the quality of reconstruction depends on the choice of $\theta$, bilevel learning can be used to optimize the parameters in $R_\theta$ \cite{Crockett_2022, ehrhardt2024optimalregularizationparametersbilevel}. Additionally, data-adaptive regularizers have shown promising results in variational regularization frameworks \cite{Yunjin_Chen_2014, InputConvex, goujon2022neuralnetworkbased}.
Moreover, components of the data fidelity term, such as the forward operator, may also depend on $\theta$. For instance, in magnetic resonance imaging (MRI), learning the sampling pattern \cite{sherry2020learning} and in seismic imaging, learning design parameters \cite{Downing_2024} directly affect the reconstruction quality.

However, large-scale bilevel optimization is challenging, particularly in imaging. Gradient-based bilevel optimization algorithms \cite{Ochs} offer scalability, but the gradient of the upper-level depends on the solution of the lower-level problem, and computing the exact solution of the lower-level variational problem is impractical. Instead, approximate gradient methods, which rely on inexact solutions of the lower-level problem, have been investigated \cite{Pedregosa1602,boundMatthias,salehi2024adaptively}. The existing imaging literature has primarily focused on methods for computing the approximate hypergradient and bounding its errors \cite{boundMatthias, piggyback_convergence, AD}, with less attention given to the behavior and convergence of the upper-level solver in the presence of such errors. Recent works have addressed these gaps by proposing adaptive bilevel methods \cite{salehi2024adaptively} or single-level reformulations \cite{suonperä2024generalsingleloopmethodsbilevel}, which reduce computational costs. However, these methods are predominantly deterministic.

In contrast, stochastic methods, particularly for the upper-level problem, can significantly accelerate computations and enable bilevel problems to be tackled on larger datasets, potentially enhancing generalization \rv{compared to methods that are limited to smaller datasets}. Most existing stochastic bilevel optimization algorithms~\cite{ghadimi2018approximation,bilevel_complexity_warmstart,ramzi2023shine} adopt a doubly-stochastic formulation, where the lower-level problem focuses on a regularized training problem%on the training dataset
, and the upper-level problem handles hyperparameter optimization on a validation dataset. However, these methods often assume a fixed number of lower-level iterations% and bounded variance in stochastic gradients
, while paying less attention to the effects of inexact lower-level evaluations or the choice of the upper-level step size. Moreover, due to only considering machine learning applications such as multinomial logistic regression, stochastic methods are considered in their lower-level problems~\cite{grazzi2021convergence}, whereas deterministic solvers are commonly used for lower-level problems arising in variational regularization~\cite{boundMatthias,salehi2024adaptively}.
%\LR{also deterministic lower-level solvers are commonly used in variational regularization? Their focus is ML applications, where lower level problem is training a model (and stochastic methods are standard)}

Another class of bilevel algorithms employs the value-function approach~\cite{dempe_bilevel_2020}, reformulating bilevel problems into a single-level constrained optimization problem, referred to as fully-first-order methods~\cite{bome,fullyFirstOrderStochastic}. While these methods theoretically provide a flexible framework for the lower-level problem, they share similar assumptions regarding 
%variance, 
a fixed number of lower-level iterations, and convergence analysis as other stochastic bilevel methods. 
% In imaging problems, the lower-level problem can often be solved effectively on small data batches~\cite{chambolle_pock_2016}, and the upper-level data can be shared with the lower-level problem. Therefore, we focus on cases where the upper-level problem is stochastic, with sampling occurring at the upper level and shared with the lower level. It is important to note that well-known stochastic methods for solving variational problems, such as SPDHG~\cite{SPDHG}, can still be employed to solve the lower-level problem to a specified tolerance, as in applications like positron emission tomography (PET) imaging. In this scenario, the stochasticity in the lower-level problem does not arise from data sampling, ensuring the validity of our discussion.

\paragraph{\textbf{Contributions:}}
In this work, we consider inexact stochastic hypergradients within the bilevel optimization framework, %calculated using the Implicit Function Theorem (IFT) combined with linear solvers (e.g., conjugate gradient), 
as outlined in \cite{Pedregosa1602,boundMatthias}. We establish the connection between these stochastic approximate hypergradients and the theoretical analysis of stochastic nonconvex optimization, particularly focusing on practical variance assumptions \cite{SGD_better,biasedSGD}. Specifically, we demonstrate that the inexact stochastic hypergradients satisfy the practical assumptions of biased stochastic gradient descent (SGD). Finally, we demonstrate the speed-up advantages of this framework compared to adaptive inexact deterministic bilevel methods in learning data-adaptive regularizers for image denoising. Additionally, we show that this approach leads to improved generalization when applied to learning data-adaptive regularizers for image deblurring problems.
\section{The Algorithm}
%\LR{
%I don't understand this section name. Maybe `bilevel learning using inexact stochastic gradient'? 
% Actually, there doesn't seem to be any mention of the lower-level problem in this section at all? You at least need to explain how the bilevel problem (with constant $\epsilon$ and $\delta$) can give you biased gradient estimators, to justify your assumption. Otherwise why are you discussing ISGD in a bilevel paper? The connection isn't clear otherwise.
% }

% Following the convention from \cite{SGD_better}, consider a sampling vector \( v \in \mathbb{R}^m \) such that \( v \sim \mathcal{D} \) and \( \mathbb{E}_\mathcal{D}[v_i] = 1 \) for all $i\in \{1,\dots,m \}$, where \( \mathcal{D} \) is a user-defined distribution (for example, the uniform distribution \(\mathcal{D} = U(0,m)\)). The upper-level optimization problem \eqref{UPPER_LEVEL} can be written as a stochastic optimization problem, and the bilevel problem takes the form below:
% \begin{subequations}\label{upper_stochastic}
% \begin{align}
%     \min_\theta \mathbb{E}_{v \sim \mathcal{D}}\Big[f_v(\theta) &\coloneqq \frac{1}{m} \sum_{i=1}^{m} v_i f_i(\theta) = \frac{1}{m}\sum_{i=1}^{m}v_ig_i(\hat{x}_i(\theta)) \Big], \label{Upper_stochastic}\\
%  s.t. \quad
%          {\hat{x}_i(\theta)} &\coloneqq \arg\min_{x\in \mathbb{R}^n} h_i(x,\theta), \quad i = 1,2,\dots, m,  \label{lower_stochastic}  
% \end{align}
% \end{subequations}
% \Sa{how to make the problem above concise?}
Considering the upper-level problem \eqref{Upper_stochastic}, for the hypergradient we can write
\begin{equation}\label{grad_upper_stochastic}
    \mathbb{E}_{v \sim \mathcal{D}}\Big[ \nabla f_v(\theta) \coloneqq \frac{1}{m} \sum_{i=1}^{m} v_i   \nabla f_i(\theta) = \frac{1}{m}\sum_{i=1}^{m}v_i \partial \hat{x}_i(\theta)^T \nabla g_i(\hat{x}_i(\theta))\Big].
\end{equation}
From the definition of the sampling vector, in case we have access to the exact lower-level solution $\hat{x}_i(\theta)$ for all $i\in \{1,2, \dots,m \}$, $f_v$ and $\nabla f_v$ are the unbiased estimators of $f$ and $\nabla f$, respectively (i.e. $\mathbb{E}_{v \sim \mathcal{D}} [f_v(\theta)] = f(\theta)$ and $\mathbb{E}_{v \sim \mathcal{D}} \nabla f_v(\theta) = \nabla f(\theta)$).
Each iteration $k \in \mathbb{N}$ of SGD takes the following form
\begin{equation}\label{iSGD_update}
    \theta^{k+1} = \theta^k - \alpha_k \nabla f_{v^k}(\theta^k),
\end{equation}
where $v^k \sim \mathcal{D}$ i.i.d. and $\alpha_k>0$ is the step size. 

Since, in general, the lower-level solution $\hat{x}_i(\theta)$ can only be approximated, one typically has access only to the approximate lower-level solution $\tilde{x}_i(\theta)$, where $\|\hat{x}_i(\theta) - \tilde{x}_i(\theta)\|\leq \epsilon$. \rv{This is ensured by the a posteriori bound $\|\nabla h_i(\tilde{x}_i(\theta), \theta)\|\leq \mu_i \epsilon$, where $\mu_i$ is the strong convexity parameter of each $h_i$.} Hence, corresponding to each sampling vector $v_i$ for all $i\in \{1,2, \dots,m \}$, the stochastic hypergradient $v_i\partial \hat{x}_i(\theta)^T \nabla g_i(\hat{x}_i(\theta))$ can be approximated using methods described in \cite{boundMatthias, AD} with an error of $\mathcal{O}(\epsilon)$. This results in an inexact stochastic hypergradient in the upper-level problem, which we denote by
 $z_v(\theta) = \nabla f_v(\theta) + e_v(\theta)$ for all $v \sim \mathcal{D}$. Note that ensuring the boundedness of this error can be done utilizing the error bounds provided in \cite{boundMatthias,grazzi20a}.  The steps of the inexact SGD algorithm are outlined in \cref{alg:isgd}. 
%  the form 
% \begin{equation}\label{ISGD_update_exact}
%     \theta^{k+1} = \theta^k - \alpha_k z_{v^k}(\theta^k),
% \end{equation}
% where $v^k \sim \mathcal{D}$ i.i.d. and $\alpha_k>0$ is the step size. This process is outlined in \cref{alg:isgd}.
\begin{algorithm}
\caption{Inexact Stochastic Gradient Descent (ISGD)}
\label{alg:isgd}
\begin{algorithmic}[1]
\State Input: Initial parameters $\theta^0 \in \mathbb{R}^d$, Sampling $\mathcal{D}$, step size $\alpha>0$, accuracy $\epsilon\geq 0$.
%\State Choose a mini-batch size $S \in \mathbb{N}$ for the stochastic gradient approximation.
\For{$k = 0,1, \dots$}
    \State{Sample a mini-batch $v^k \sim \mathcal{D}$.}
    \State{Compute the approximate stochastic gradient $z_{v^k}(\theta^k)$ with error $\mathcal{O}(\epsilon)$.}
    \State{Update the iterate: $\theta^{k+1} = \theta^k - \alpha  z_{v^k}(\theta^k)$}\label{update_step}
\EndFor
\end{algorithmic}
\end{algorithm}
\subsection{Convergence Analysis}
To prove convergence of Algorithm \ref{alg:isgd}, we consider the notations below and make the following regularity assumptions.
% \begin{definition}\label{epsilon_def}
%     For any $v \sim \mathcal{D}$ and for all \(\theta \in \mathbb{R}^d\),  the error of the hypergradient takes the form \( \epsilon_v(\theta) = \|e_v(\theta)\| = \|\nabla f_v(\theta) - z_v(\theta) \|  \).
% \end{definition}
\begin{notation*}For any $v \sim \mathcal{D}$ and for all \(\theta \in \mathbb{R}^d\),  the error of the hypergradient is shown as $\|e_v(\theta)\|$. Moreover, at each iteration $k \in \mathbb{N}$ of \cref{alg:isgd}, we denote $z_k \coloneqq z_{v^k}(\theta^k)$, and $\epsilon_k^2 = \E_{v^k\sim \mathcal{D}} [\|e_{v^k}(\theta^k)\|^2]$.
\end{notation*}
\begin{assumption}\label{assumption1} Each $f_i$ is $L_{\nabla f_i}$-smooth, which means $f$ and the $f_i$'s are continuously differentiable with $L_{\nabla f}$ and $L_{\nabla f_i}$ Lipschitz gradients, respectively. Moreover, each $f_i$ is bounded below.
\end{assumption}

% \begin{assumption}\label{assumption2}
%     (Bounded error) 
%     There exists $\bar{\epsilon} \geq 0$ such that for all \(\theta \in \mathbb{R}^d\), \( \E_{v \sim D}[\|e_v(\theta)\|] \leq \bar{\epsilon} \) (i.e. the error of inexact stochastic gradient is uniformly bounded). %Note that the inequality holds with probability $1$.
% \end{assumption}

%%%%%%%%%%%%%
% \begin{theorem}\label{convergence}
%     \textbf{Convergence:} Let assumptions \ref{assumption1} and \ref{assumption2} hold. Let $\delta_0 \overset{\text{def}}{=} f(\theta^0) - f^*$, and choose the stepsize such that $0 < \alpha \leq \frac{b(\zeta)}{L_{\nabla f} \tilde{B}(\eta)}$. Then the iterates $\{\theta^t\}_{t \geq 0}$ of \eqref{ISGD_update_exact} satisfy
% \begin{align*}
% \min_{0 \leq t \leq T-1} \mathbb{E} \left [\|\nabla f(\theta^t)\|^2 \right ] \leq \frac{2(1 + L_{\nabla f} \tilde{A}(\eta)\alpha^2)}{b(\zeta)\alpha T} \delta_0 &+ \frac{(1+\eta)L_{\nabla f}(C + \eta \E[\epsilon_t^2])\alpha}{\eta b(\zeta)} \\ &+\frac{\zeta \E[\epsilon_t^2]}{2b(\zeta)}.
% \end{align*}
% \end{theorem}
\begin{theorem}\label{convergence}
    \textbf{Convergence:} Let \cref{assumption1} hold \rv{and $\E[\epsilon_k^2]$ be bounded above for all $k\geq 0$}. Let $\delta_0 \overset{\text{def}}{=} f(\theta^0) - f^*$. There exist constants $c_1,c_2,\dots,c_5>0$ such that if the step size satisfies $0 < \alpha \leq \frac{c_1}{L_{\nabla f}}$, then the iterates $\{\theta^t\}_{t \geq 0}$ of %\eqref{ISGD_update_exact}
    \cref{alg:isgd} (step \ref{update_step}) satisfy
\begin{align*}
\min_{0 \leq t \leq T-1} \mathbb{E} \left [\|\nabla f(\theta^t)\|^2 \right ] \leq \frac{ c_2 \rv{(1+ c_3 L_{\nabla f} \alpha^2)^T}}{\alpha T} \delta_0 + c_4L_{\nabla f}\alpha + c_5.
\end{align*}
\end{theorem}
%%%%%%%%%%%%%
%%corollary
% \begin{corollary}\label{cor:1}
%     There exists $\alpha>0$ and $T_0\in \mathbb{N}$, such that for all $T\geq T_0$ we have
%     \begin{equation}\label{converge_bound}
% \min_{0 \leq t \leq T-1} \mathbb{E} \left [\|\nabla f(\theta^t)\|^2 \right ] \leq \frac{3\zeta \E[\epsilon_T^2]}{2b(\zeta)}.
% \end{equation}
% \end{corollary}
\begin{corollary}\label{cor:1}
    For \rv{any} fixed accuracy $\delta>0$, there exists \rv{$\alpha_\delta>0$} and $T_\delta \in \mathbb{N}$, such that for all $T\geq T_\delta$ we have
    \begin{equation}\label{converge_bound}
\min_{0 \leq t \leq T-1} \mathbb{E} \left [\|\nabla f(\theta^t)\|^2 \right ] \leq \mathcal{O}(\delta).
\end{equation}
\rv{Moreover, $T_\delta = \mathcal{O}(\delta^{-2})$ and $\alpha_\delta = \mathcal{O}(\delta)$.}
\end{corollary}
%%%%%%%%%%%%%
Recent stochastic optimization analyses for nonconvex objectives have highlighted a more general, flexible, and practical bounded variance assumption, introducing the ABC assumption~\cite{gower2019sgd,SGD_better,biasedSGD}. However, in the stochastic bilevel framework, the role of inexact stochastic hypergradients and their connection to this framework remains unexplored. We establish this connection by stating the following lemma.
\begin{lemma}\label{lemma_ABC}
     \cite[Proposition 2]{SGD_better} Expected Smoothness (ES): Let Assumption \ref{assumption1} hold and  $\mathbb{E}[v_i^2]<\infty$ for $1\leq i\leq m$. There exist constants $A, B, C\geq 0$ such that for all $\theta \in \mathbb{R}^d$, the stochastic gradient satisfies 
     \begin{equation*}
         \mathbb{E}_{v \sim \mathcal{D}}[ \|\nabla f_{v}(\theta)\|^2] \leq 2A ( f(\theta) - f^{\text{inf}} )+ B \|\nabla f(\theta)\|^2 + C.
     \end{equation*}
\end{lemma}
% \begin{proof}
%     See \cite[Proposition 2]{SGD_better}.
% \end{proof}

The following lemmas will be utilized in proving the convergence of ISGD. \rv{The strategy is to demonstrate the connection between the inexact stochastic gradient estimator and the biased ABC condition, enabling the use of convergence results under the biased ABC condition.}
\begin{lemma}\label{Young'sLemma}
    For all $\theta \in \mathbb{R}^d$, $v \sim \mathcal{D}$, and $\zeta>0$ we have
    \begin{equation}\label{Young's}
        |\langle \nabla f(\theta), \E_{v\sim \mathcal{D}}[e_v(\theta)]\rangle| \leq \frac{\|\nabla f(\theta)\|^2}{2\zeta} + \frac{\zeta}{2} (\E_{v \sim D}[\|e_v(\theta)\|])^2%(\E_{v\sim \mathcal{D}} [\epsilon_v(\theta)])^2.
    \end{equation}
\end{lemma}
\begin{proof}
    For all $\theta \in \mathbb{R}^d$ and $v \sim \mathcal{D}$, utilizing Cauchy--Schwarz inequality we find
    $$|\langle \nabla f(\theta), \E_{v\sim \mathcal{D}}[e_v(\theta)]\rangle| \leq \|\nabla f(\theta) \| \| \E_{v\sim \mathcal{D}}[e_v(\theta)]\|.$$
    Now, since $\|\cdot \|$ is convex, using Jensen's inequality we have
    $$\| \E_{v\sim \mathcal{D}}[e_v(\theta)]\| \leq \E_{v\sim \mathcal{D}} \left[\|e_v(\theta)\|\right].$$
Combining these inequalities we find %and \cref{assumption2}, there exists  $\bar{\epsilon}\geq 0$ such that 
$$|\langle \nabla f(\theta), \E_{v\sim \mathcal{D}}[e_v(\theta)]\rangle| \leq \|\nabla f(\theta) \| \E_{v \sim D}[\|e_v(\theta)\|].$$
Applying Young's inequality on $\|\nabla f(\theta) \| \E_{v \sim D}[\|e_v(\theta)\|]$, with a parameter $\zeta>0$ yields
$$\|\nabla f(\theta) \| \E_{v \sim D}[\|e_v(\theta)\|] \leq \frac{\|\nabla f(\theta)\|^2}{2\zeta} + \frac{\zeta}{2}(\E_{v \sim D}[\|e_v(\theta)\|])^2,$$
which implies \eqref{Young's}.
\end{proof}

The following lemma relates the second moment of the inexact stochastic gradient with the stochastic gradient.
\begin{lemma}\label{second_moment_connect}
Let $\eta>0$. The following relation between the second moment of the exact and inexact stochastic gradients hold
\begin{equation}\label{i_e_expectation}
    \E_{v \sim \mathcal{D}}[\|z_v(\theta)\|^2] \leq \frac{1+\eta}{\eta} \E_{v \sim \mathcal{D}} [\|\nabla f_v(\theta)\|^2] + (1+\eta) \E_{v \sim \mathcal{D}}[\|e_v(\theta)\|^2].
\end{equation}
\end{lemma}
 \begin{proof}
     From the definition of $z_v(\theta)$, for all $\theta \in \mathbb{R}^d$, we have
     \begin{equation}\label{expanded_i}
         \E_{v \sim \mathcal{D}}[\|z_v(\theta)\|^2] = \E_{v \sim \mathcal{D}}[\|\nabla f_v(\theta)\|^2] + \E_{v \sim \mathcal{D}}[\|e_v(\theta)\|^2] + 2\E_{v \sim \mathcal{D}}[\langle \nabla f_v(\theta), e_v(\theta) \rangle].
     \end{equation}
 Now, by applying Cauchy--Schwarz and then Young's inequality on the last term above, with a parameter $\eta>0$, we get
 \begin{align*}
     2\E_{v \sim \mathcal{D}}[\langle \nabla f_v(\theta), e_v(\theta) \rangle] &\leq 2\E_{v \sim \mathcal{D}}[\|\nabla f_v(\theta)\| \|e_v(\theta)\|] \\ &\leq \frac{1}{\eta}\E_{v \sim \mathcal{D}}[\|\nabla f_v(\theta)\|^2] + \eta \E_{v \sim \mathcal{D}}[\|e_v(\theta)\|^2].
 \end{align*}
 % $$2\E_{v \sim \mathcal{D}}[\langle \nabla f_v(\theta), e_v(\theta) \rangle] \leq 2\E_{v \sim \mathcal{D}}[\|\nabla f_v(\theta)\| \|e_v(\theta)\|] \leq \frac{1}{\eta}\E_{v \sim \mathcal{D}}[\|\nabla f_v(\theta)\|^2] + \eta \E_{v \sim \mathcal{D}}[\|e_v(\theta)\|^2].$$
 Plugging it back in \eqref{expanded_i} yields
  \eqref{i_e_expectation} as desired.
 \end{proof}

% \begin{remark}
%     \textcolor{red}{(Not practically interesting)} The optimal $\eta$ which provides the tightest upper-bound in \eqref{i_e_expectation} is $\eta = \frac{\sqrt{\E_{v \sim \mathcal{D}} [\|\nabla f_v(\theta)\|^2]}}{\epsilon}$.
% \end{remark}

\begin{proposition}\label{ES_IBES}
    Suppose Assumption \ref{assumption1} \rv{hold and $\mathbb{E}[v_i^2]<\infty$ for $1\leq i\leq m$}. Let $\eta >0$, $\zeta > 0$, $\epsilon^2(\theta) \coloneqq \E_{v\sim \mathcal{D}} [\|e_v(\theta)\|^2]$, and  $b(\zeta) = 1 - \frac{1}{2\zeta}$, $c(\zeta, \epsilon^2(\theta)) = \frac{\zeta}{2}\epsilon^2(\theta)$, $\tilde{A}(\eta) = \frac{1+\eta}{\eta} A$,
    $\tilde{B}(\eta) = \frac{1+\eta}{\eta} B$,
    $\tilde{C}(\eta, \epsilon^2(\theta)) = \frac{1+\eta}{\eta} C + (1+\eta)\epsilon^2(\theta)$.
    % \begin{itemize}
    % \item $b(\zeta) = 1 - \frac{1}{2\zeta}$, 
    % \item $c(\zeta, \epsilon^2(\theta)) = \frac{\zeta}{2}\epsilon^2(\theta)$,
    %     \item $\tilde{A}(\eta) = \frac{1+\eta}{\eta} A$,
    %     \item $\tilde{B}(\eta) = \frac{1+\eta}{\eta} B$,
    %     \item $\tilde{C}(\eta, \epsilon^2(\theta)) = \frac{1+\eta}{\eta} C + (1+\eta)\epsilon^2(\theta)$.
    % \end{itemize}
    Then, for all $\theta \in \mathbb{R}^d$, the inexact stochastic gradient estimator satisfies 
     \begin{align}
        \langle \nabla f(\theta), \E_{v\sim \mathcal{D}}[z_v(\theta)]\rangle &\geq b(\zeta)\|\nabla f(\theta)\|^2 -c(\zeta, \epsilon^2(\theta)),\label{biasineq} \\
         \mathbb{E}_{v \sim \mathcal{D}}[ \|z_{v}(\theta)\|^2] &\leq 2\tilde{A}(\eta) ( f(\theta) - f^{\text{inf}} )+ \tilde{B}(\eta) \|\nabla f(\theta)\|^2 + \tilde{C}(\eta, \epsilon^2(\theta))\label{tildeABC}.
     \end{align}
\end{proposition}

\begin{proof}
For the inexact stochastic gradient $z_v(\theta)$, we have 
    \begin{equation*}
        \langle \nabla f(\theta), \mathbb{E}_{v \sim \mathcal{D}}[z_v(\theta)]\rangle = \langle \nabla f(\theta), \mathbb{E}_{v \sim \mathcal{D}}[\nabla f_v(\theta)] + \mathbb{E}_{v \sim \mathcal{D}} [e_v(\theta)]\rangle.
    \end{equation*}
Since the estimator of the exact gradient is unbiased i.e. $\E_{v\sim \mathcal{D}}[\nabla f_v(\theta)] = \nabla f(\theta)$, we have
\begin{equation}\label{expand_inner_i_e}
    \langle \nabla f(\theta), \mathbb{E}_{v \sim \mathcal{D}}[z_v(\theta)]\rangle = \|\nabla f(\theta)\|^2 + \langle \nabla f(\theta), \mathbb{E}_{v \sim \mathcal{D}}[e_v(\theta)]\rangle.
\end{equation}
From \cref{Young'sLemma} we know $\langle \nabla f(\theta), \mathbb{E}_{v \sim \mathcal{D}}[e_v(\theta)]\rangle \geq - \frac{\|\nabla f(\theta)\|^2}{2\zeta} - \frac{\zeta}{2}(\E_{v\sim \mathcal{D}} [\|e_v(\theta)\|])^2$. Combining it with \eqref{expand_inner_i_e} yields
$$ \langle \nabla f(\theta), \mathbb{E}_{v \sim \mathcal{D}}[z_v(\theta)]\rangle \geq (1 - \frac{1}{2\zeta}) \|\nabla f(\theta) \|^2 -  \frac{\zeta}{2}(\E_{v\sim \mathcal{D}} [\|e_v(\theta)\|])^2.$$
Now utilizing Jensen's inequality $\E_{v \sim \mathcal{D}}[\|e_v(\theta)\|^2] \geq (\E_{v \sim \mathcal{D}}[\|e_v(\theta)\|])^2$ we find
$$
\langle \nabla f(\theta), \mathbb{E}_{v \sim \mathcal{D}}[z_v(\theta)]\rangle \geq (1 - \frac{1}{2\zeta}) \|\nabla f(\theta) \|^2 -  \frac{\zeta}{2}\epsilon^2(\theta),
$$
which gives us $b(\zeta)$ and $c(\zeta, \epsilon^2(\theta))$ as required. 
From \cref{lemma_ABC}, there exist $A,B,C \geq 0$ such that 
    \begin{equation}\label{ABC}
        \mathbb{E}_{v \sim \mathcal{D}}[ \|\nabla f_{v}(\theta)\|^2] \leq 2 {A} ( f(\theta) - f^{\text{inf}} )+ {B}  \|\nabla f(\theta)\|^2 + {C}.
    \end{equation}
    Utilizing \cref{second_moment_connect}, for all $\eta>0$ we have    
    \begin{equation*}
    \E_{v \sim \mathcal{D}}[\|z_v(\theta)\|^2] \leq \frac{1+\eta}{\eta} \E_{v \sim \mathcal{D}} [\|\nabla f_v(\theta)\|^2] + (1+\eta) \E_{v \sim \mathcal{D}}[\|e_v(\theta)\|^2].
\end{equation*}
Using the upper bound of $\mathbb{E}_{v \sim \mathcal{D}}[ \|\nabla f_{v}(\theta)\|^2]$ from \eqref{ABC} in the inequality above yields
\begin{equation}\label{tildeABC_precise}
    \E_{v \sim \mathcal{D}}[\|z_v(\theta)\|^2] \leq 2\frac{1+\eta}{\eta} A ( f(\theta) - f^{\text{inf}} ) + \frac{1+\eta}{\eta} B  \|\nabla f(\theta)\|^2 + \frac{1+\eta}{\eta} C + (1+\eta)\epsilon^2(\theta).
\end{equation}
Comparing \eqref{ABC} and \eqref{tildeABC_precise}, we find the required $\tilde{A}(\eta)$, $\tilde{B}(\eta)$, and $\tilde{C}(\eta, \epsilon^2(\theta))$.
\end{proof}

\begin{remark}
 Since \cite[Proposition 2 and 3]{SGD_better} which states there exist certain $A$, $B$, $C\geq 0$ for each sampling strategy, holds under the same assumptions as \cref{ES_IBES}, using the connection between $\tilde{A}$, $\tilde{B}$, and $\tilde{C}$, and ${A}$, ${B}$, and ${C}$, \cite[Proposition 2 and 3]{SGD_better} holds for the inexact stochastic gradient.
\end{remark}

% \LR{where is the citation for the below results? You should say explicitly somewhere that Proposition 1 shows you meet the assumptions of the biased SGD paper, and so the below comes directly from there}
% \begin{remark}
%     Noting that \cref{ES_IBES} demonstrates, under Assumptions \ref{assumption1} and \ref{assumption2}, that the inexact stochastic gradient in our bilevel setting satisfies the biased ABC condition \cite[Assumption 9]{biasedSGD}, the following convergence theorem for biased SGD directly applies to ISGD.
% \end{remark}
Now, utilizing the proposition above, we discuss the proofs of \cref{convergence} and \cref{cor:1}.
\begin{proof}[\textbf{Proof of \cref{convergence}}]
    Since \cref{ES_IBES} demonstrates, under \cref{assumption1}, that the inexact stochastic gradient estimator in our bilevel setting satisfies the biased ABC condition \cite[Assumption 9]{biasedSGD}, the convergence theorem for biased SGD \cite[Theorem 3]{biasedSGD} directly applies to ISGD. Hence, utilising \cref{ES_IBES} by setting $c_1 = \frac{b(\zeta)}{\tilde{B}(\eta)}$, $c_2 = \frac{2}{b(\zeta)} $, $c_3 = {\tilde{A}(\eta)} $, $c_4 = \frac{(1+\eta)C}{\eta b(\zeta)}+ \frac{1+\eta}{b(\zeta)}\rv{\tau}$, $c_5 = \frac{\zeta}{b(\zeta)}\rv{\tau}$, \rv{where $\tau\geq0$ such that $\sum_{k=0}^T \E[\epsilon_k^2]\leq \tau$}, $\zeta>\frac{1}{2}$ and $\eta>0$, \cref{convergence} holds.
\end{proof}
% \Sa{I moved the remark to the proof of theorem but could not fully resolve Matthias' comments: This is your theorem and Peter's. Sequence (6) is iSGD. Then the proof is using his + your other stuff earlier. I also suggest to have the convergence statement much earlier and then prove it.}
\begin{proof}[\textbf{Proof of \cref{cor:1}}]\label{proof:cor}
    By setting \begin{equation*}
T(\E[\epsilon_T^2]) = \frac{12\delta_0 L_{\nabla f}}{\zeta\E[\epsilon_T^2]}\max \left\{\frac{\tilde{B}(\eta)}{b(\zeta)}, \frac{12\delta_0 \tilde{A}(\eta)}{\zeta\E[\epsilon_T^2]}, \frac{2(1+\eta)(C +\eta \E[\epsilon_T^2])}{\eta \zeta\E[\epsilon_T^2]}\right\}, \ \text{and }
\end{equation*}
\begin{equation*}
    \rv{\alpha(\E[\epsilon_T^2])} = \min \left\{ \frac{1}{\sqrt{L_{\nabla f}\tilde{A}(\eta)T}}, \frac{b(\zeta)}{L_{\nabla f}\tilde{B}(\eta)}, \frac{\eta \zeta\E[\epsilon_T^2]}{2 L_{\nabla f}(1+\eta)(C +\eta \E[\epsilon_T^2])} \right\},
\end{equation*}
by using \cref{convergence} with constants $c_1 = \frac{b(\zeta)}{\tilde{B}(\eta)}$, $c_2 = \frac{2}{b(\zeta)} $, $c_3 = {\tilde{A}(\eta)} $, $c_4 = \frac{(1+\eta)C}{\eta b(\zeta)}+ \frac{1+\eta}{b(\zeta)}\rv{\tau}$, $c_5 = \frac{\zeta}{b(\zeta)}\rv{\tau}$, \rv{where $\tau\geq0$ such that $\sum_{k=0}^T \E[\epsilon_k^2]\leq \tau$}, $\zeta>\frac{1}{2}$ and $\eta>0$, for $T\geq T(\E[\epsilon_T^2])$ we have 
\begin{equation*}
   \min_{0 \leq t \leq T-1} \mathbb{E} \left [\|\nabla f(\theta^t)\|^2 \right ] \leq \frac{3\zeta \E[\epsilon_T^2]}{2b(\zeta)}. 
\end{equation*}
Now for a fixed accuracy $\delta>0$, substituting $\delta = \E[\epsilon_T^2]$, setting $T_\delta \coloneqq T(\delta)$, and \rv{$\alpha_\delta \coloneqq \alpha(\delta)$}, we find the desired result.
\end{proof}
\begin{remark}
    Note that if $\{\epsilon_t\}_{t=0}^\infty$ is chosen as a deterministic non-negative sequence such that $\sum_{t=0}^\infty \epsilon_t^2 < \infty$, and $\zeta\geq \frac{1}{2}$, then following similar steps as the the proof of \cref{cor:1}, we have 
    $\lim_{T\to\infty} \big( \min_{0 \leq t \leq T-1} \mathbb{E} \left [\|\nabla f(\theta^t)\|^2 \right ] \big) = 0$.
\end{remark}
% \begin{corollary}
% Choose the stepsize $\alpha > 0$ as 
% $$
% \alpha = \min \left\{ \frac{1}{\sqrt{L_{\nabla f}\tilde{A}(\eta)T}}, \frac{b(\zeta)}{L_{\nabla f}\tilde{B}(\eta)}, \frac{\eta \zeta\E[\epsilon_t^2]}{2 L_{\nabla f}(1+\eta)(C +\eta \E[\epsilon_t^2])} \right\}.
% $$
% Then, if $$ T \geq \frac{12\delta_0 L_{\nabla f}}{\zeta\E[\epsilon_t^2]}\max \left\{\frac{\tilde{B}(\eta)}{b(\zeta)}, \frac{12\delta_0 \tilde{A}(\eta)}{\zeta\E[\epsilon_t^2]}, \frac{2(1+\eta)(C +\eta \E[\epsilon_t^2])}{\eta \zeta\E[\epsilon_t^2]}\right\},$$
% we have
% \begin{equation*}
% \min_{0 \leq t \leq T-1} \mathbb{E} \left [\|\nabla f(\theta^t)\|^2 \right ] \leq \frac{3\zeta \E[\epsilon_t^2]}{2b(\zeta)}.
% \end{equation*}
% Interestingly, if $\epsilon^2(\theta) \to 0$, then $$\left ( \min_{0 \leq t \leq T-1} \mathbb{E} \left [\|\nabla f(\theta^t)\|^2 \right ] \right )\to 0.$$
% By employing Jensen's inequality $\mathbb{E} \left [\|\nabla f(\theta^t)\| \right ]^2 \leq \mathbb{E} \left [\|\nabla f(\theta^t)\|^2 \right ]$, we get
% $$\lim_{\epsilon^2(\theta) \to 0}\big (\min_{0 \leq t \leq T-1} \mathbb{E} \left [\|\nabla f(\theta^t)\|\right ] \big) = 0.$$
% \begin{equation*}
%     \min_{0 \leq t \leq T-1} \mathbb{E} \left [\|\nabla f(\theta^t)\|\right ] \leq \frac{3\zeta \E[\epsilon_t^2]}{2b(\zeta)}.
% \end{equation*}
% Note that if $\{\epsilon_t\}_{t=0}^\infty$ is chosen as a deterministic sequence that $\sum_{t=0}^\infty \epsilon_t^2 < \infty$, then right had side above converges to zero.
% \end{corollary}
\section{Numerical Experiments}
\rv{The implementation codes are available on the GitHub repository\footnote{\url{https://doi.org/10.5281/zenodo.14987053}}.}
\label{sec:numeric}
\subsection{Learning Field of Experts Regularizer for Denoising}
In this section, we consider Field of Experts \cite{FoE} (FoE) priors as a data-adaptive approach. \rv{In all experiments, accelerated gradient descent \cite[Algorithm 3]{chambolle_pock_2016} is used as the lower-level solver, while the conjugate gradient method is used to approximate the inverse of the Hessian of the lower-level objective involved in computing the hypergradient.} Given $m$ pairs $\{(y_i, x_i^*)\}_{i=1}^m$ of noisy images with Gaussian noise and clean images, the bilevel problem correspond to learning the FoE regularizer for image denoising takes the following form %\vspace{-10pt}
\begin{equation}\label{denoising_foe}
    \min_{\theta \in \mathbb{R}^d}  \frac{1}{m}\sum_{i=1}^{m}\|\hat{x}_i(\theta) - x_i^*\|^2 \quad s.t. \quad
         {\hat{x}_i(\theta)} = \arg\min_{x\in \mathbb{R}^n} \frac{1}{2}\|x - y_i\|^2 + R_\theta(x),
          %\vspace{-10pt}
\end{equation}
% \begin{subequations}\label{denoising_foe}
%     \begin{align}
%         \min_{\theta \in \mathbb{R}^d} & \frac{1}{m}\sum_{i=1}^{m}\|\hat{x}_i(\theta) - x_i^*\|^2 \label{upper_denoise}\\ s.t. \quad
%          {\hat{x}_i(\theta)} =& \arg\min_{x\in \mathbb{R}^n} \frac{1}{2}\|x - y_i\|^2 + R_\theta(x), \quad i = 1,\dots, m,\label{lower_denoise}
%     \end{align}
% \end{subequations}
where $R_\theta(x) = e^{\theta_0}\sum_{j = 1}^J e^{\theta_j} \|c_j \ast x\|_{\theta_{J+j}}$, each $c_j$ represents a $2D$ convolutional kernel with parameters dependent on $\theta$, and $\|x\|_\nu = \sum_{i=1}^N\sqrt{x_i^2+\nu^2} - \nu$ is the smoothed $1$-norm. We take $J=10$ kernels of size $7\times7$ with $3$ channels, which yields to $d = 1482$ parameters. Moreover we chose $m = 1024$ colored training images of size $96\times 96$ from STL10 dataset\footnote{\url{https://cs.stanford.edu/\textasciitilde acoates/stl10/}} \rv{in mini-batches of size $S=32$}, and to obtain each $y_i$, we add Gaussian noise drawn from $\mathcal{N}(0, \sigma)$ with noise level $25$ (i.e. $\sigma = \frac{25}{255}$).
% \begin{figure}[H] 
%     \centering
%     \includegraphics[width=0.8\textwidth]{SSVM/Figs/denoising_STL/loss_total_iter.png} 
%     \caption{Average loss per epoch vs total lower-level computational cost (iterations of lower-level solver plus iterations of linear system solver}
%     \label{fig:loss_total_iter}
% \end{figure}
\begin{figure}[htbp] 
    \centering
    % First subfigure
    \begin{subfigure}[b]{0.49\textwidth}
        \centering
        \includegraphics[width=\textwidth]{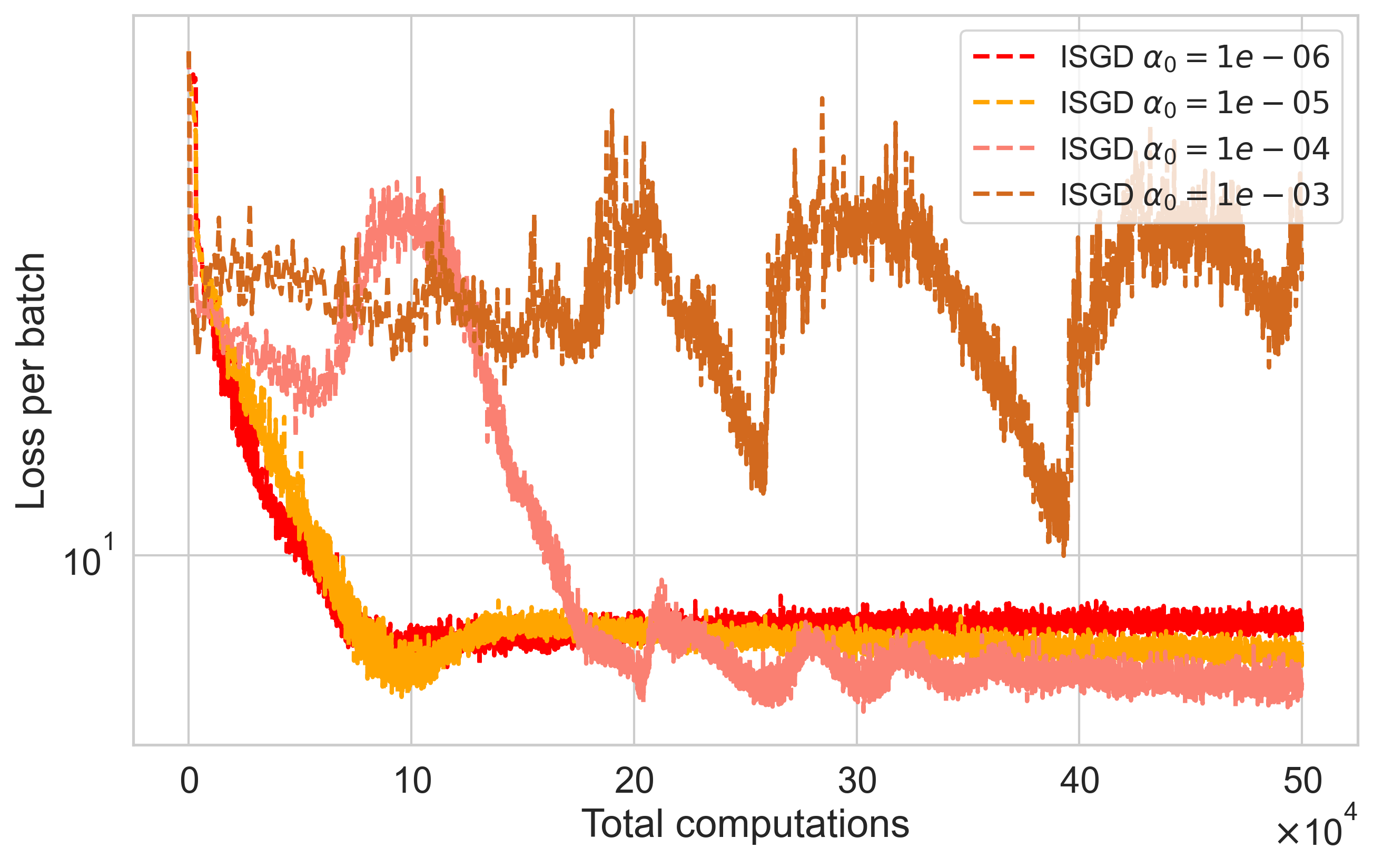}

        \label{fig:loss_total_iter_fixed}
    \end{subfigure}
    \hfill
    % Second subfigure
    \begin{subfigure}[b]{0.49\textwidth}
        \centering
        \includegraphics[width=\textwidth]{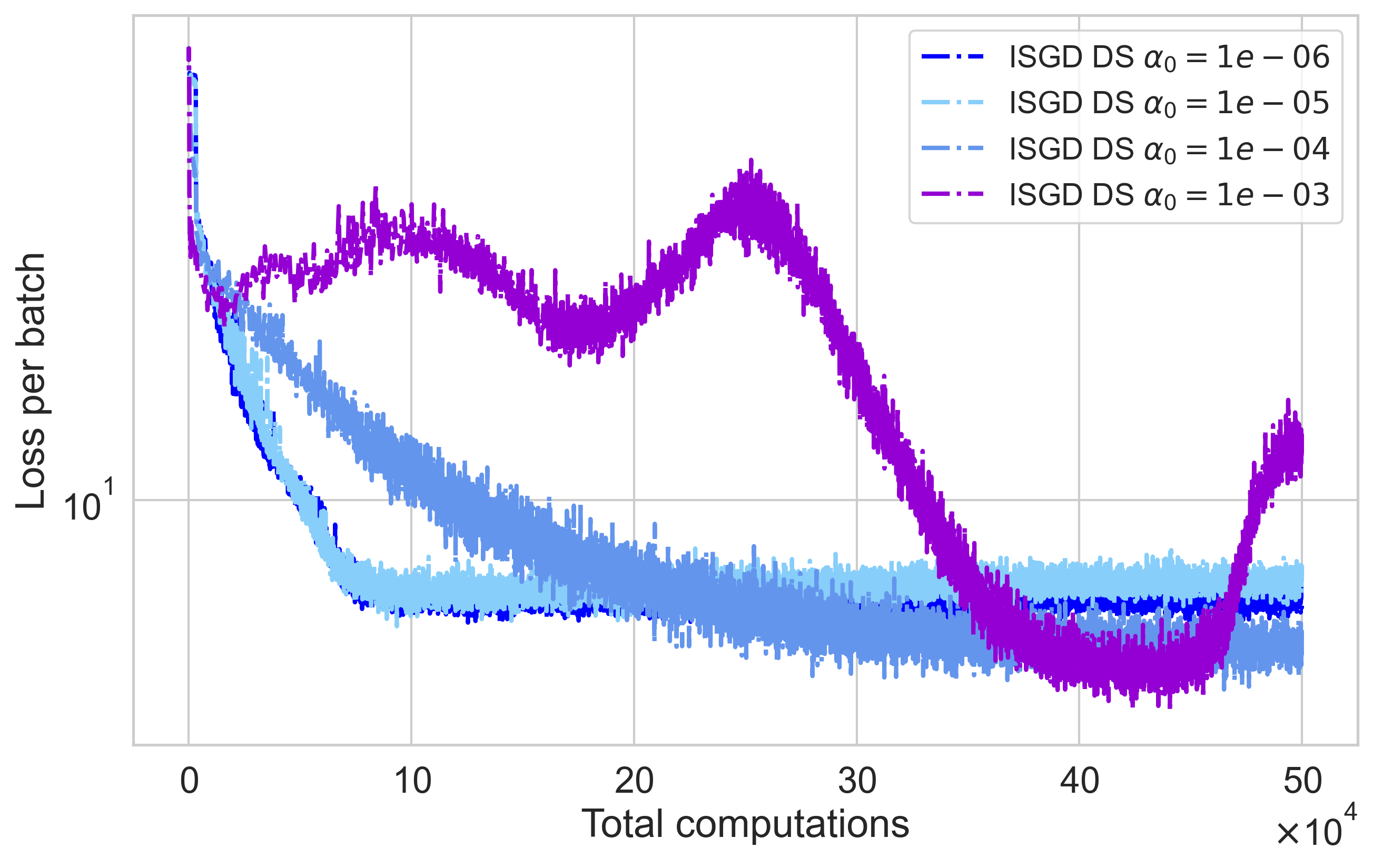}
     
        \label{fig:loss_total_iter_DS}
    \end{subfigure}
    \caption{\footnotesize{Loss per mini-batch versus total lower-level computational cost for ISGD with a constant step size, \rv{$\alpha_k = \alpha_0$}, \rv{(left)}, and a decreasing step size \rv{(DS)}, $\alpha_k = \frac{\alpha_0}{\sqrt{k}}$, \rv{(right)}. While the decreasing step size variant exhibits more stability, a fixed small step size achieves comparable performance.}}
    \label{fig:loss_total_iter_combined}
    \vspace{-15pt}
\end{figure}
In \cref{fig:loss_total_iter_combined}, the left subfigure shows the loss per iteration for ISGD with fixed step sizes plotted against the total lower-level computations for solving \eqref{denoising_foe}. The right subfigure presents the same results for ISGD with decreasing step sizes $\alpha_k = \frac{\alpha_0}{\sqrt{k}}$. For fixed $\alpha_0 = 10^{-3}$, oscillations were high without sustained loss reduction, while the decreasing step size version reduced oscillations, achieving the lowest loss temporarily before increasing at the end. Fixed $\alpha_0 = 10^{-4}$ initially fluctuated but eventually achieved the lowest loss among fixed settings. Its decreasing variant was more stable early on, eventually matching the fixed version’s success without diverging. Both $\alpha_0 = 10^{-5}$ and $10^{-6}$ showed strong initial loss reduction, plateauing after $10^5$ computations, with decreasing variants demonstrating similar trends but with less fluctuation. Overall, smaller step sizes provide stable initial decreases but limited long-term progress, whereas larger step sizes, especially with decreasing schedules, improve long-term performance, underscoring the importance of their theoretical analysis.
\begin{figure}[htbp]
    \centering
    % First subfigure
    \begin{subfigure}[b]{0.49\textwidth}
        \centering
        \includegraphics[width=\textwidth]{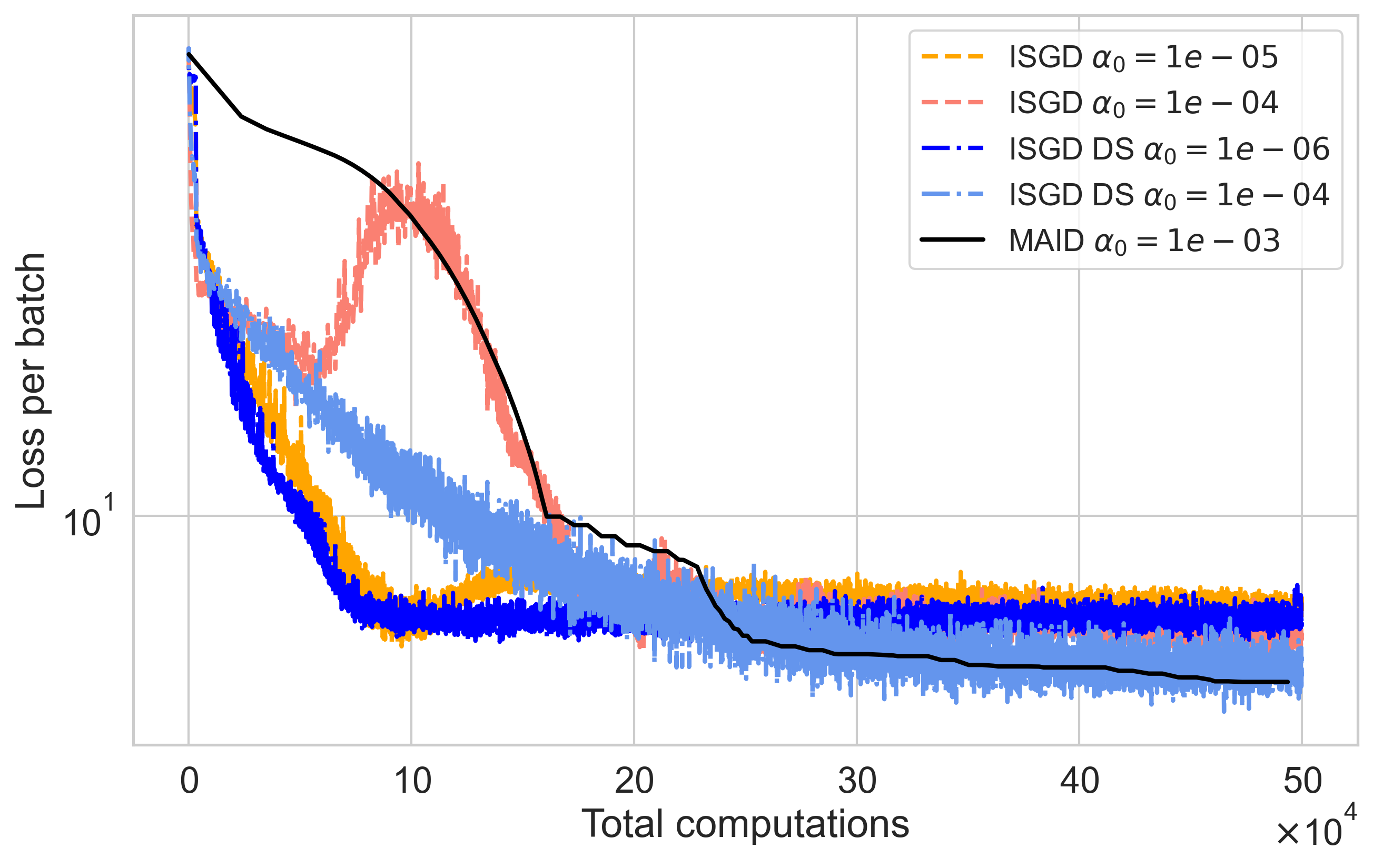}

        \label{fig:loss_total_iter_mixed}
    \end{subfigure}
    \hfill
    % Second subfigure
    \begin{subfigure}[b]{0.49\textwidth}
        \centering
        \includegraphics[width=\textwidth]{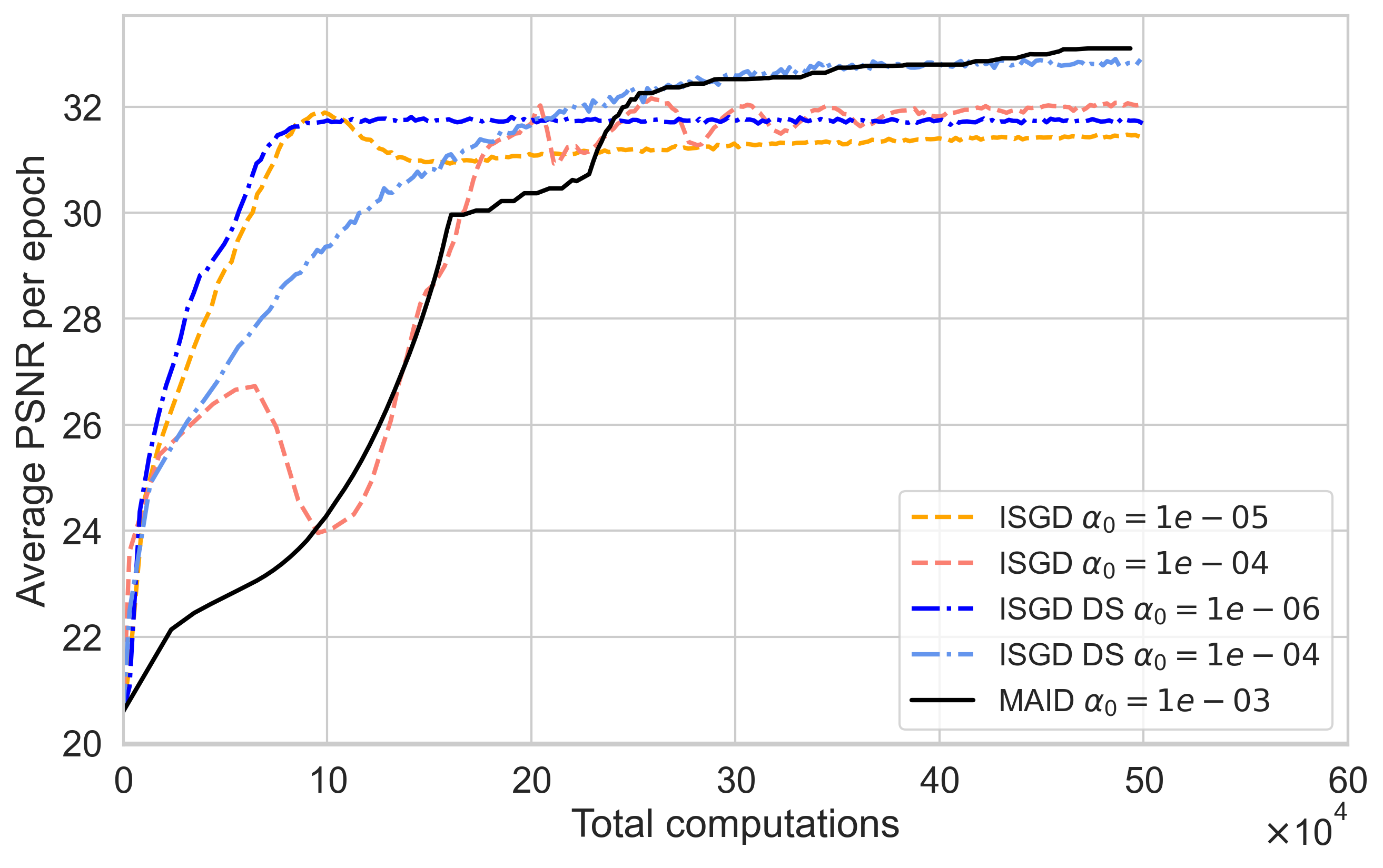}
     
        \label{fig:psnr_ISGD_MAID}
    \end{subfigure}
    \caption{\footnotesize{Comparison of ISGD (with and without decreasing step sizes) and MAID in terms of loss per mini-batch for ISGD variants and upper-level loss for MAID, plotted against computational cost, as well as average training PSNR per epoch vs total computations.}}
    \label{fig:loss_total_iter_ISGD_MAID}
    \vspace{-15pt}
\end{figure}
In \cref{fig:loss_total_iter_ISGD_MAID}, the left subfigure compares the upper-level loss captured per batch for ISGD variants and per epoch for Method of Adaptive Inexact Descent (MAID) \cite{salehi2024adaptively}, plotted against total lower-level computations. \rv{The hyperparameters of MAID are selected based on empirical tuning.} ISGD variants, except fixed step size $\alpha_0 = 10^{-4}$, reached a neighborhood of the local minima faster than MAID within $2\times 10^5$ computations. However, MAID’s adaptivity allowed it to perform comparably to the best ISGD variant ($\alpha_0 = 10^{-4}$ with decreasing step size) by the end. The right subfigure shows a similar trend for average PSNR per epoch, demonstrating ISGD’s faster performance on the training dataset. Up to $2\times 10^5$ computations, ISGD achieved significantly higher PSNR than MAID, highlighting its speed advantage during early training stages.

To qualitatively demonstrate this difference, \cref{fig:checkpoints} compares the performance of the best ISGD variants (fixed step size $\alpha_0 = 10^{-5}$, and $\alpha_0 = 10^{-4}$ with decreasing step size) to MAID on a training image. Before $2\times 10^5$ total computations, ISGD achieves a higher PSNR.
\begin{figure}[htbp]
    \centering
    %%%%%GT
      \begin{minipage}[b]{0.155\linewidth}
        \centering
        \includegraphics[width=\linewidth]{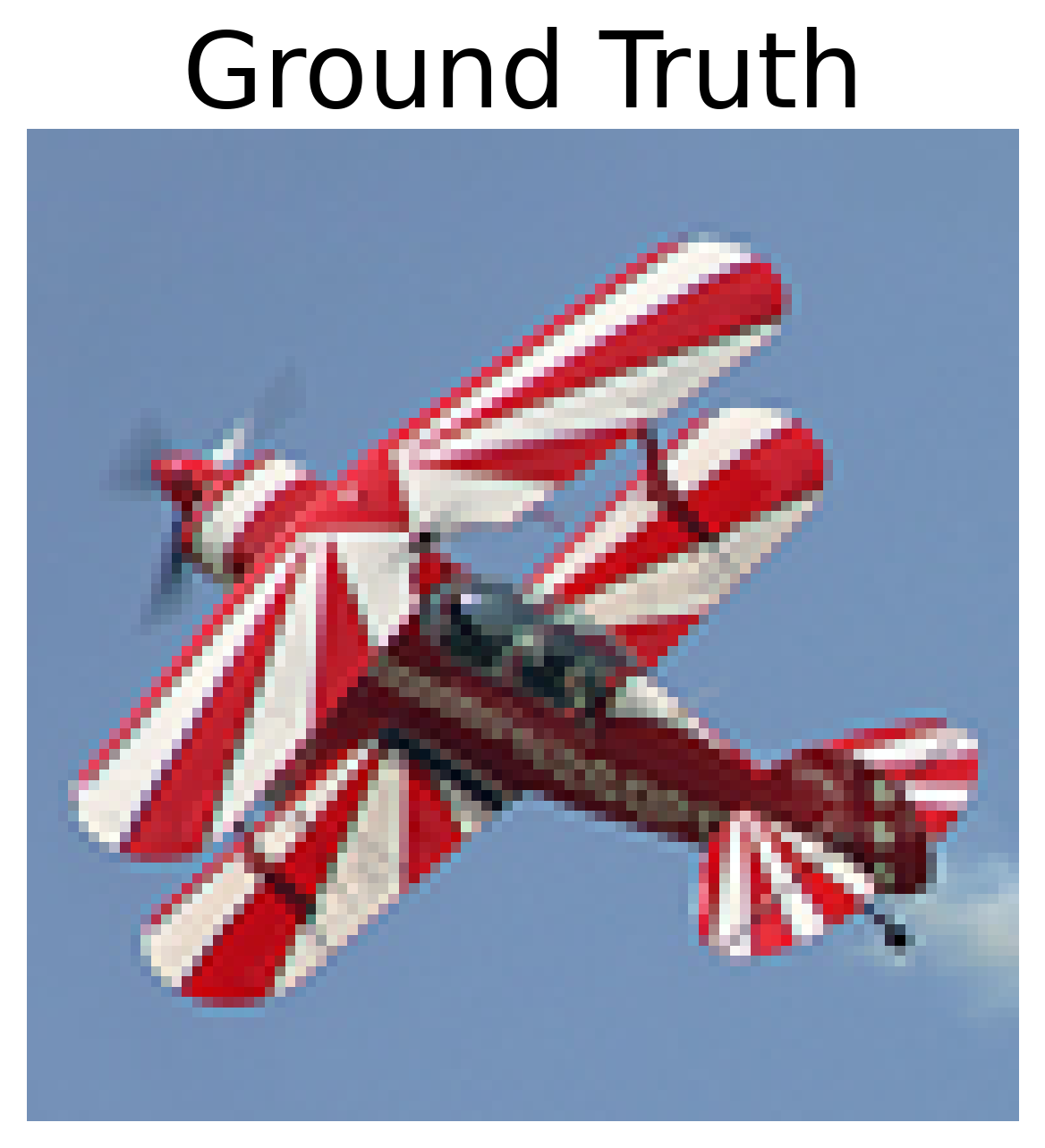}
        % \caption*{\footnotesize Ground truth}
    \end{minipage}
    %%%%%
    \hspace{40pt}
    \begin{minipage}{0.05\linewidth}
    \centering
    \vspace{-45pt}
    \rotatebox{90}{\rv{\footnotesize ISGD}}
    \end{minipage}
\hfill
    \begin{minipage}[b]{0.155\linewidth}
    \centering
    \includegraphics[width=\linewidth]{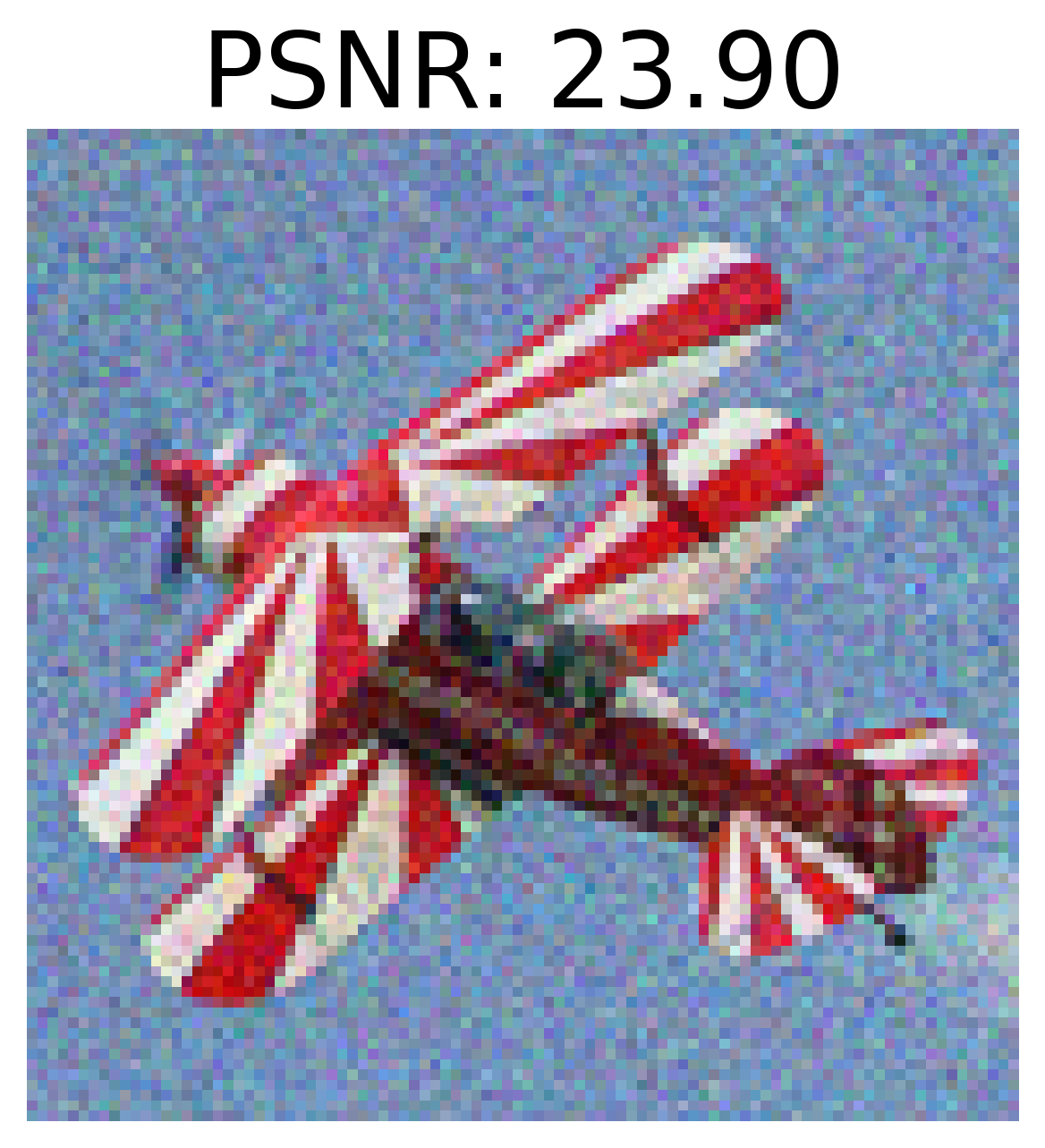}
\end{minipage}
\hfill
\begin{minipage}[b]{0.155\linewidth}
    \centering
    \includegraphics[width=\linewidth]{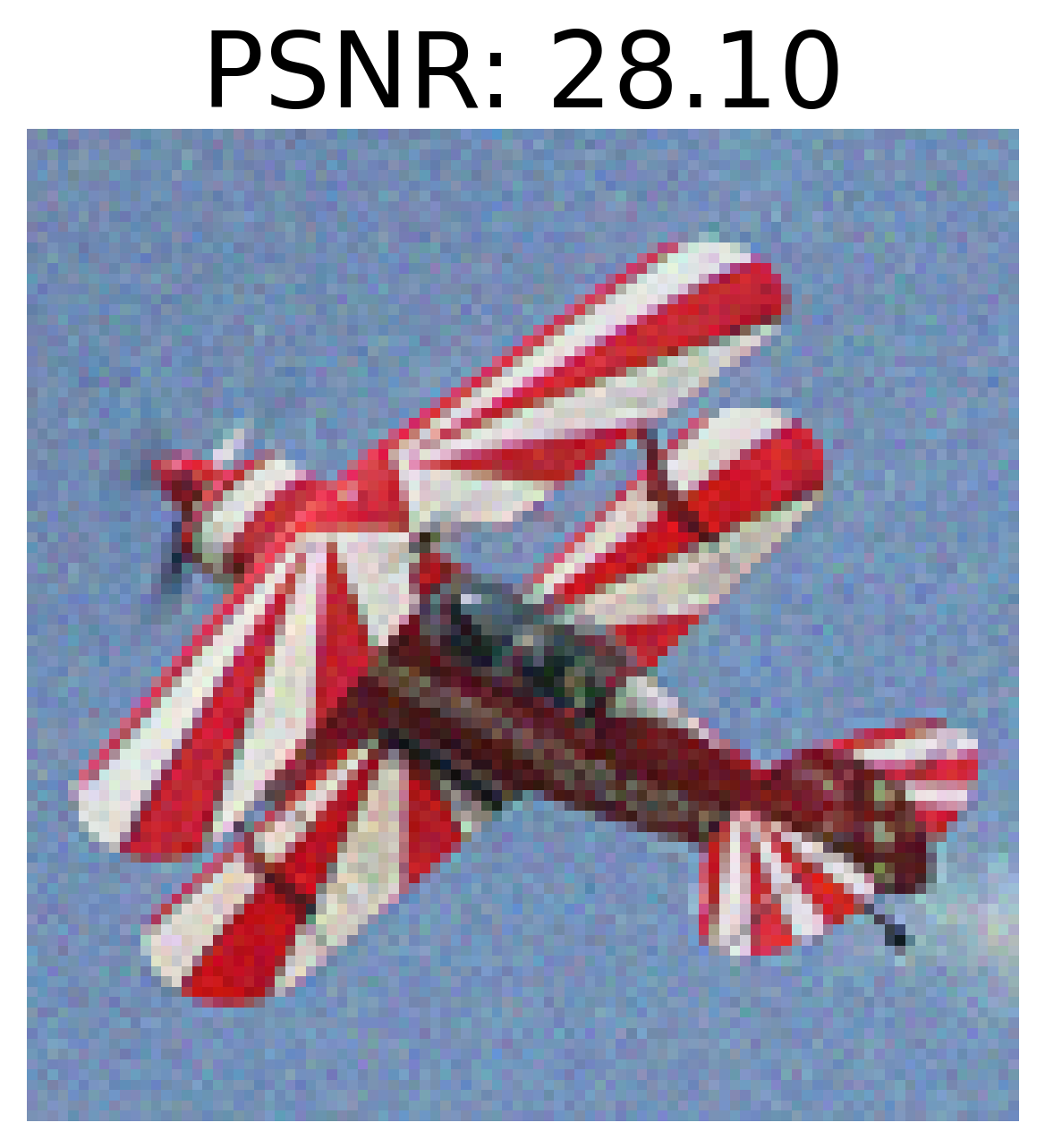}
\end{minipage}
\hfill
\begin{minipage}[b]{0.155\linewidth}
    \centering
    \includegraphics[width=\linewidth]{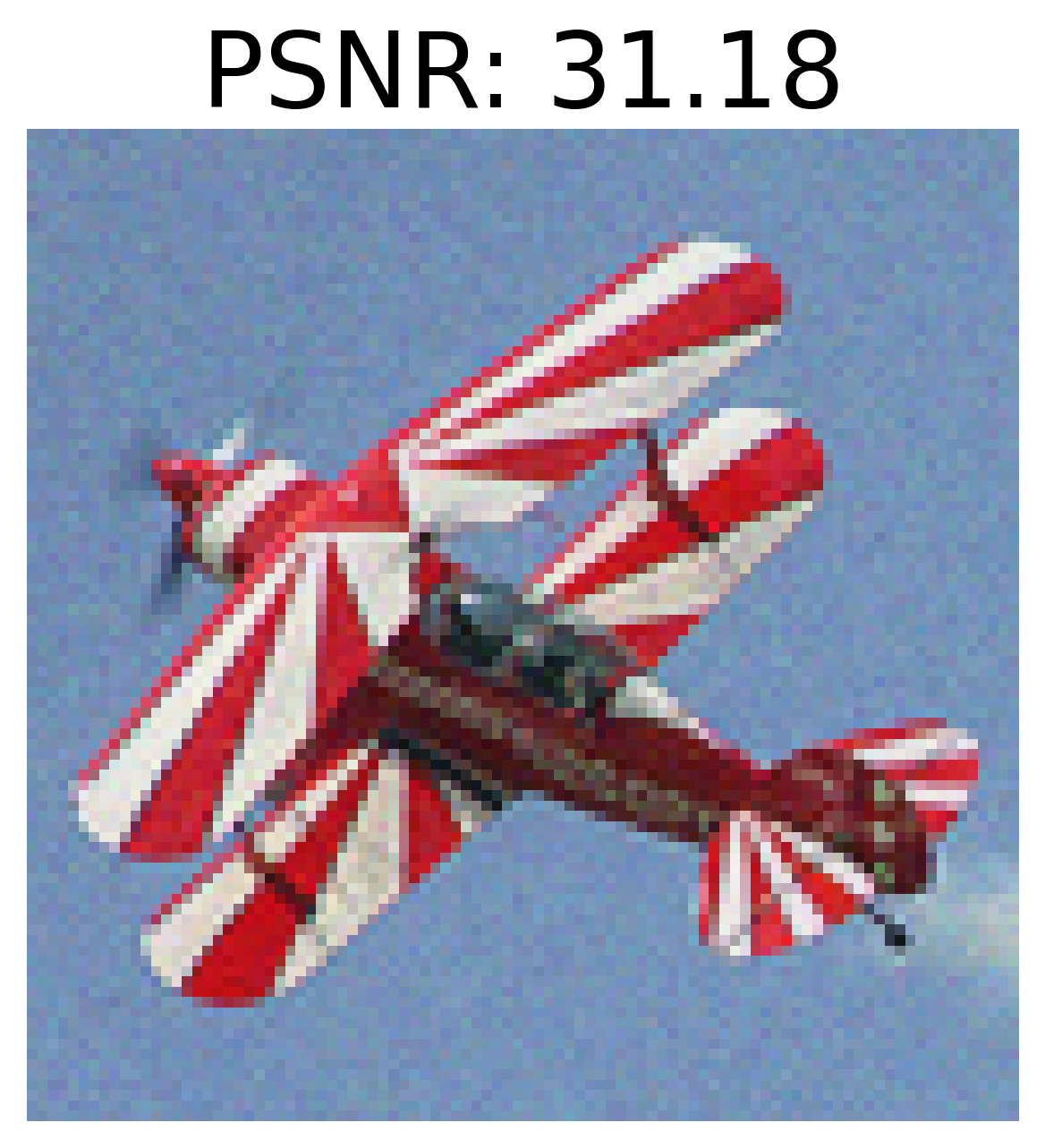}
\end{minipage}
\hfill
\begin{minipage}[b]{0.155\linewidth}
    \centering
    \includegraphics[width=\linewidth]{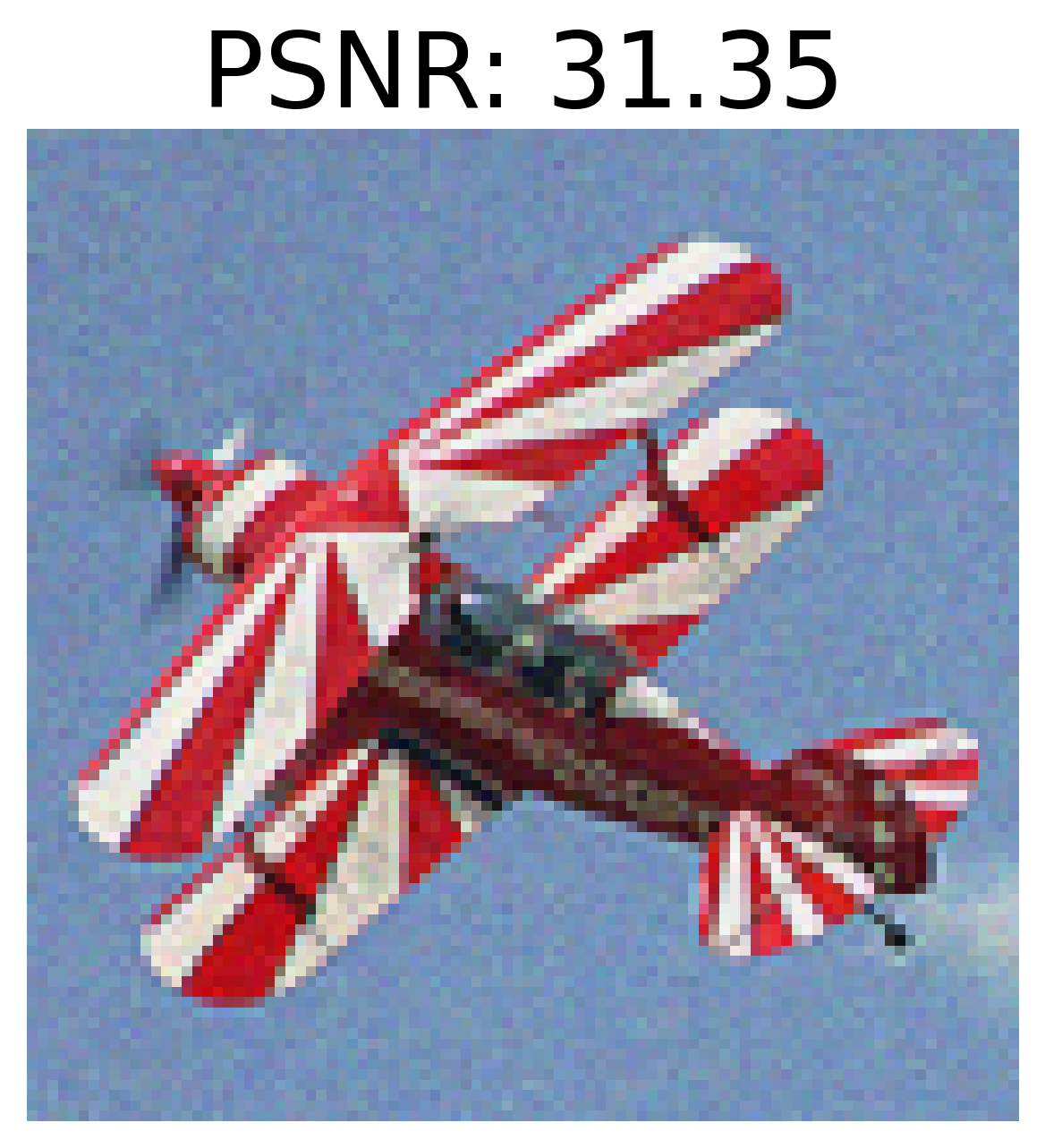}
\end{minipage}
    % \hfill
     \vspace{2pt}
     %%%second row
     % 
\begin{minipage}[b]{0.155\linewidth}
        \centering
        \includegraphics[width=\linewidth]{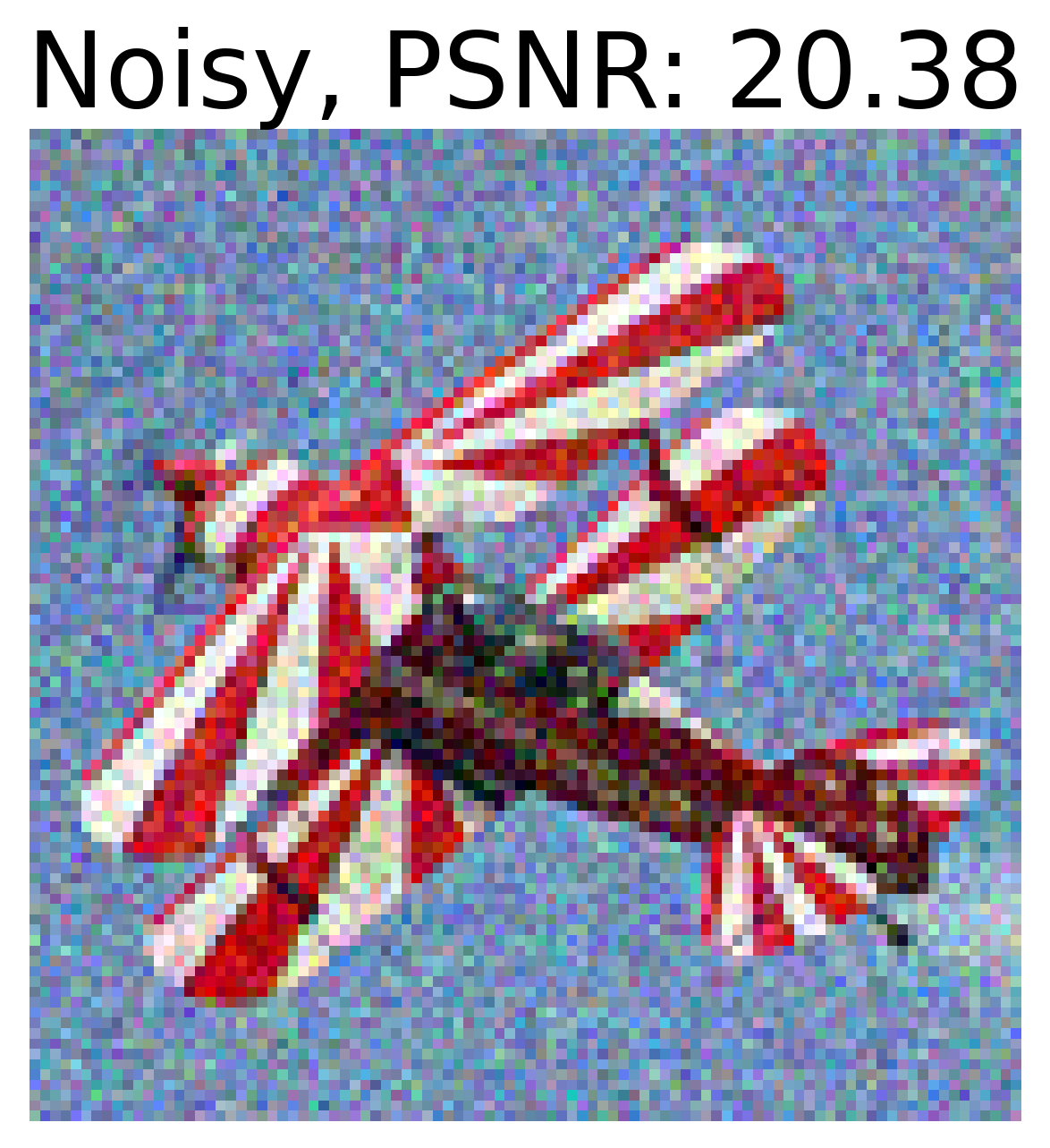}
    \end{minipage}
% \hfill
\hspace{40pt}
    \begin{minipage}{0.05\linewidth}
    \centering
    \vspace{-45pt}
    \rotatebox{90}{\rv{\footnotesize ISGD DS}}
    \end{minipage}
\hfill
\begin{minipage}[b]{0.155\linewidth}
        \centering
        \includegraphics[width=\linewidth]{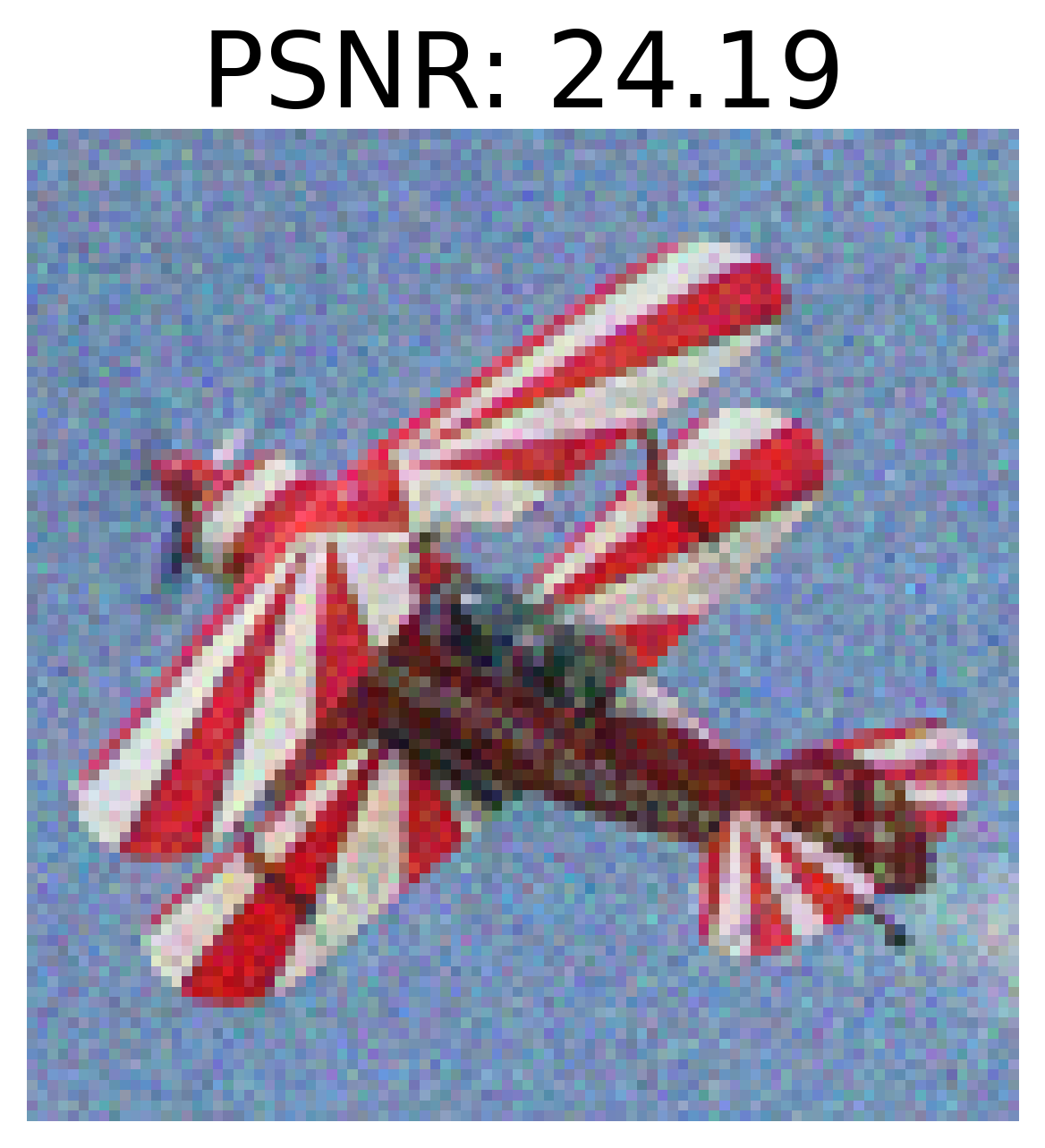}
    \end{minipage}
    \hfill
    \begin{minipage}[b]{0.155\linewidth}
        \centering
        \includegraphics[width=\linewidth]{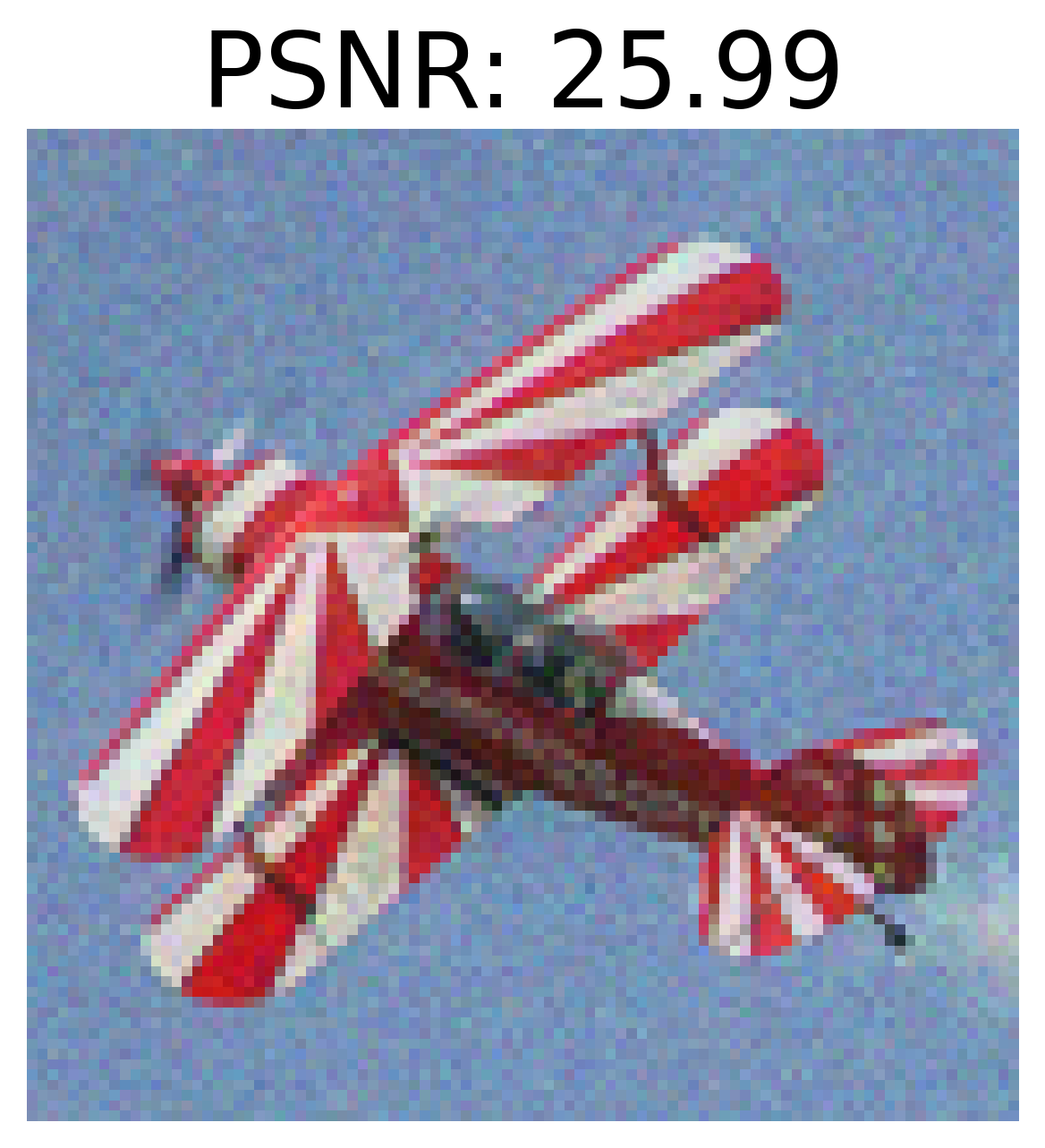}
    \end{minipage}
    \hfill
    \begin{minipage}[b]{0.155\linewidth}
        \centering
        \includegraphics[width=\linewidth]{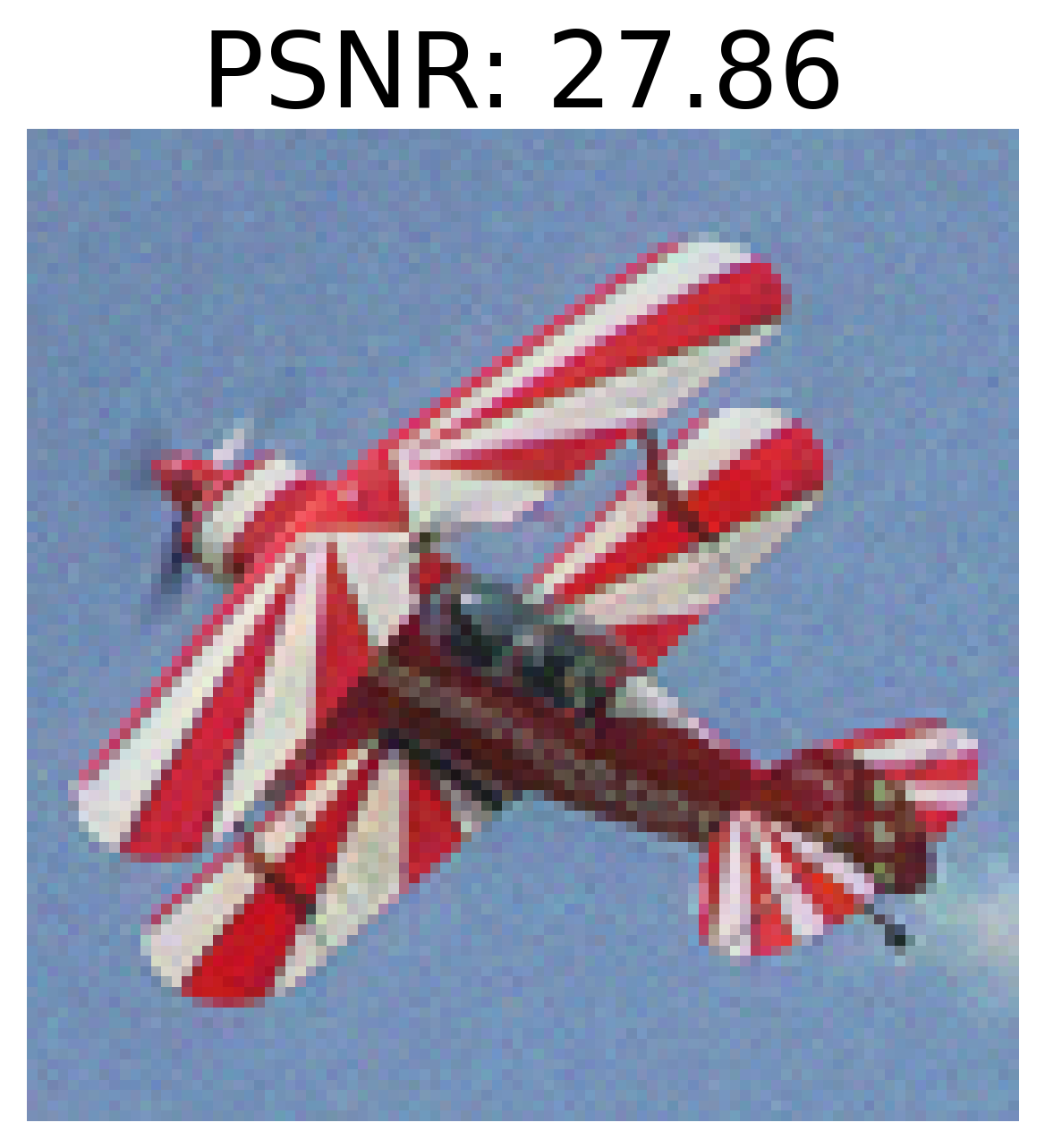}
    \end{minipage}
    \hfill
    \begin{minipage}[b]{0.155\linewidth}
        \centering
        \includegraphics[width=\linewidth]{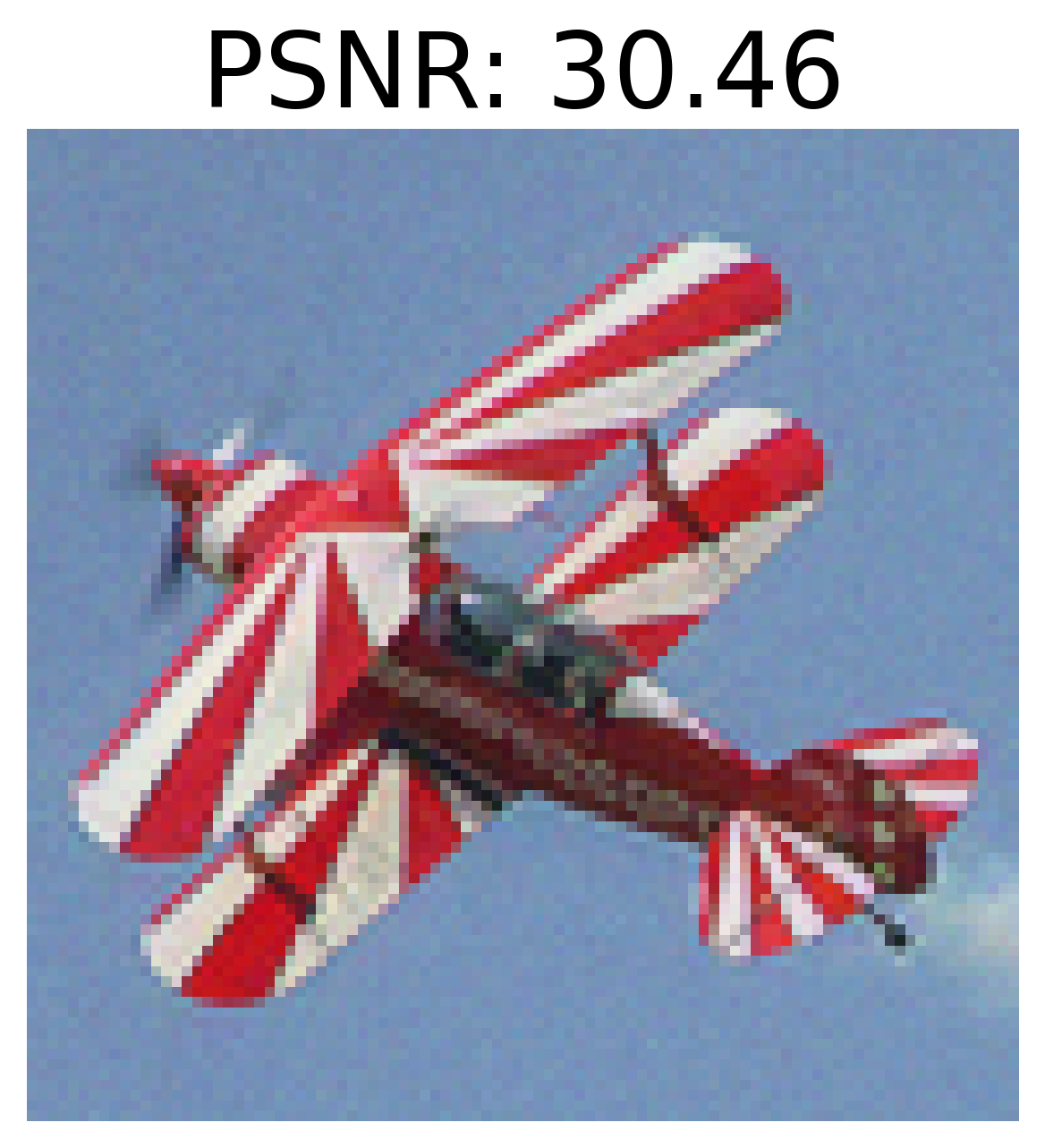}
    \end{minipage}
    % third row of images
    %%%%Noisy image
    \begin{minipage}[b]{0.155\linewidth}
    \centering
    \caption*{ } % This ensures the first figure slot is intentionally blank
\end{minipage}
    % \hfill
    \hspace{40pt}
     \begin{minipage}{0.05\linewidth}
    \centering
    \vspace{-90pt}
    \rotatebox{90}{\rv{\footnotesize MAID}}
    \end{minipage}
    \hfill
    %%%%
    \begin{minipage}[b]{0.155\linewidth}
        \centering
        \includegraphics[width=\linewidth]{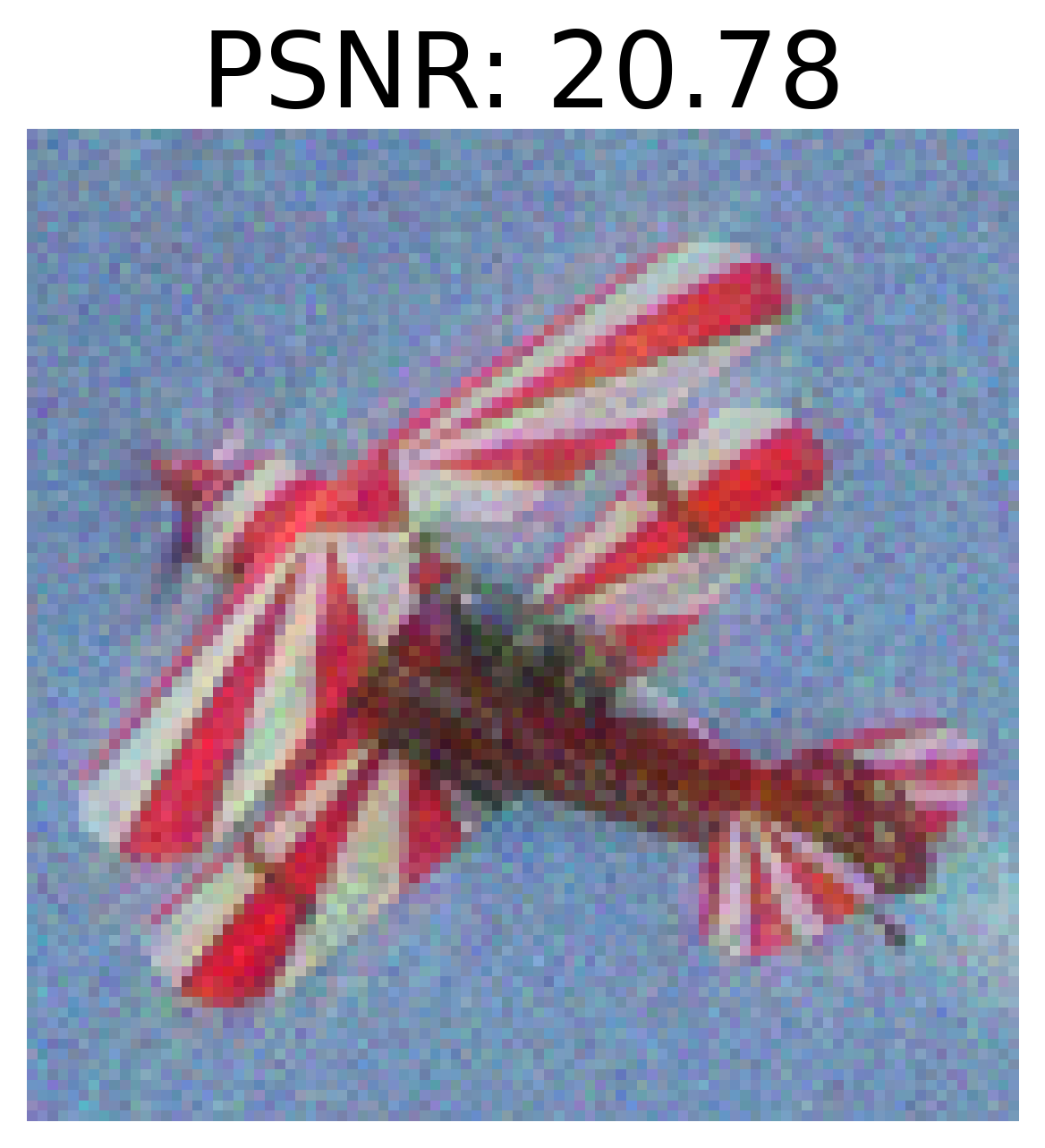}
        \caption*{\footnotesize $10^4$}
    \end{minipage}
    \hfill
    \begin{minipage}[b]{0.155\linewidth}
        \centering
        \includegraphics[width=\linewidth]{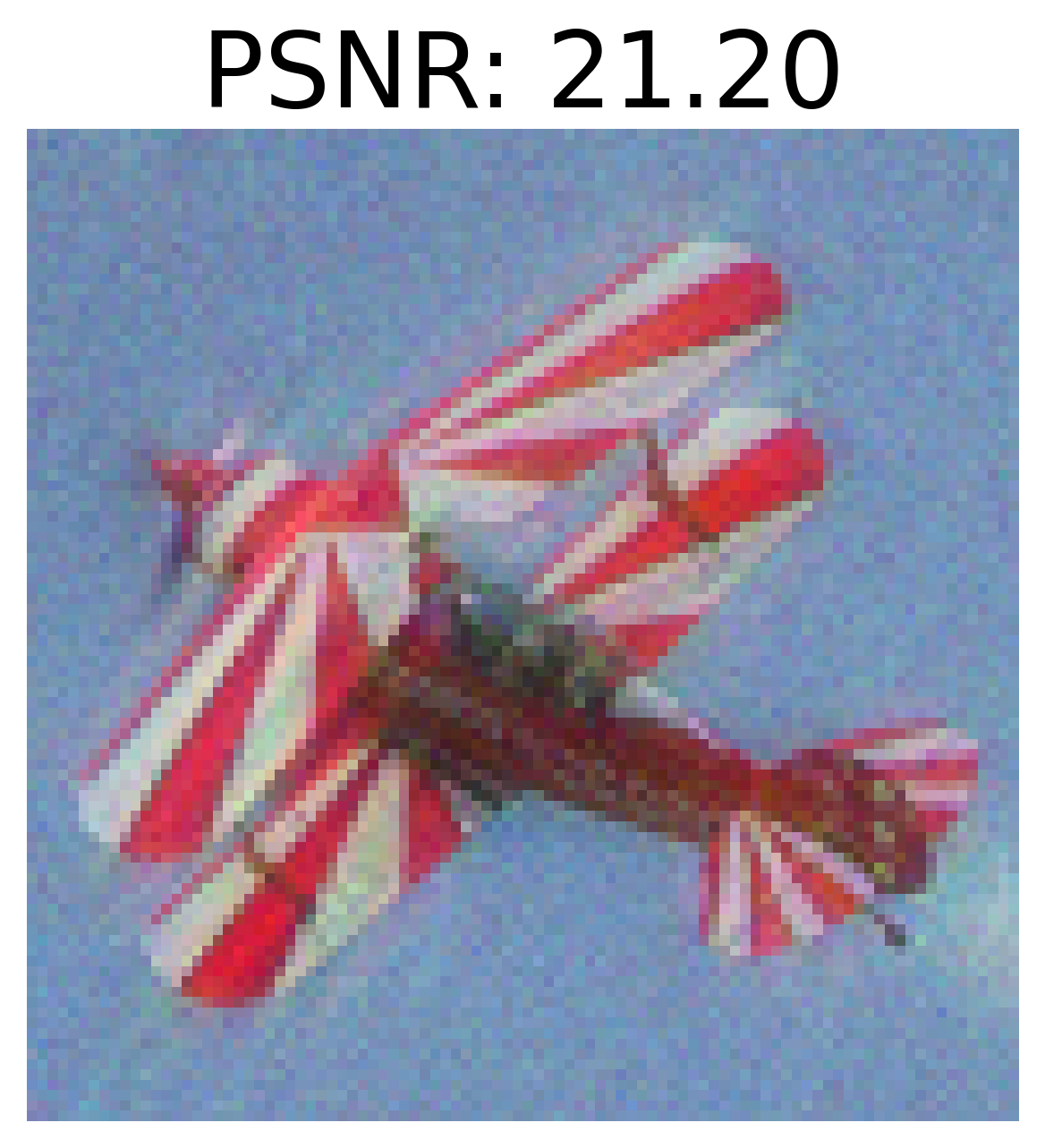}
        \caption*{\footnotesize $5 \times 10^4$}
    \end{minipage}
    \hfill
    \begin{minipage}[b]{0.155\linewidth}
        \centering
        \includegraphics[width=\linewidth]{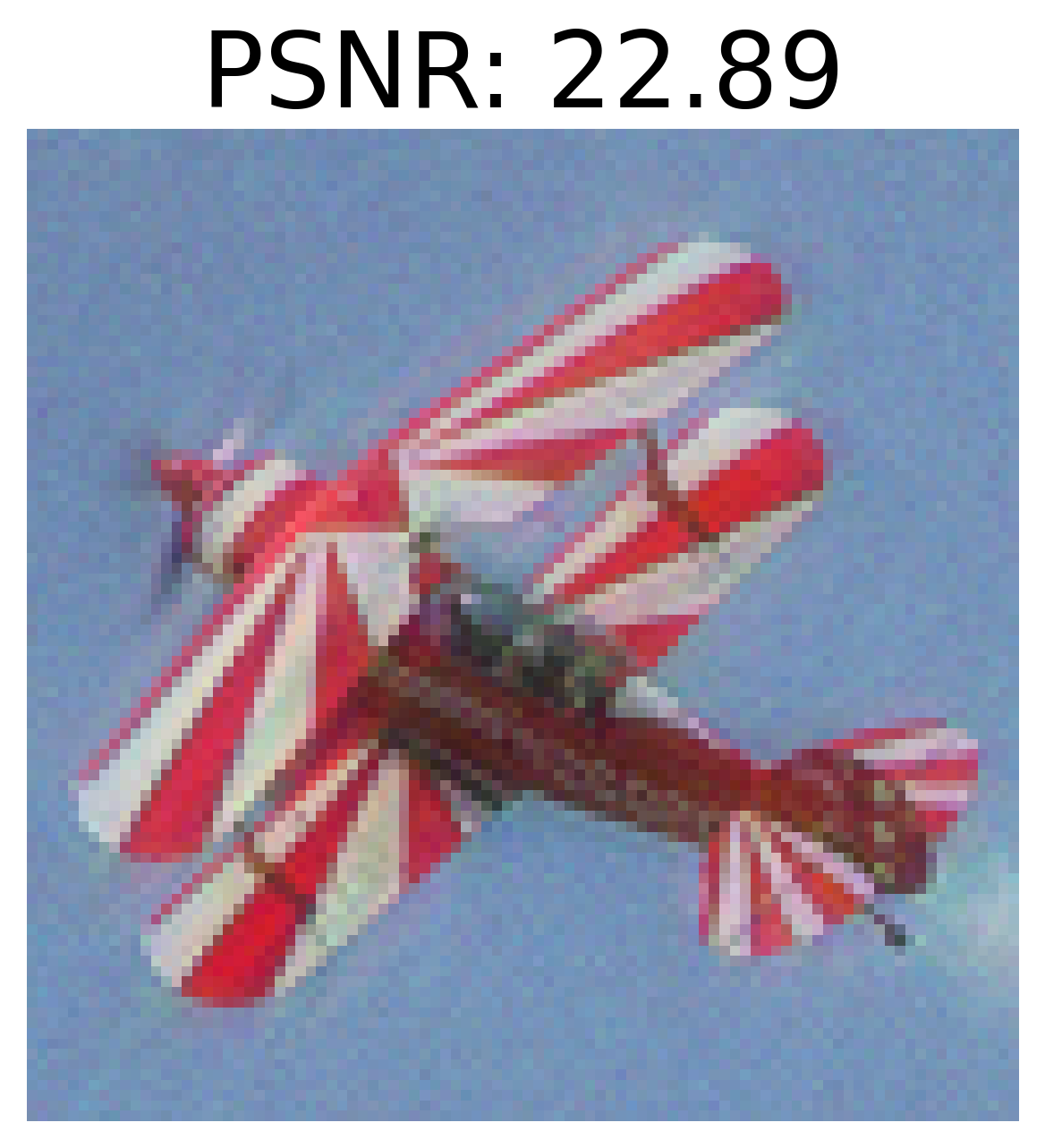}
        \caption*{\footnotesize $10^5$}
    \end{minipage}
    \hfill
    \begin{minipage}[b]{0.155\linewidth}
        \centering
        \includegraphics[width=\linewidth]{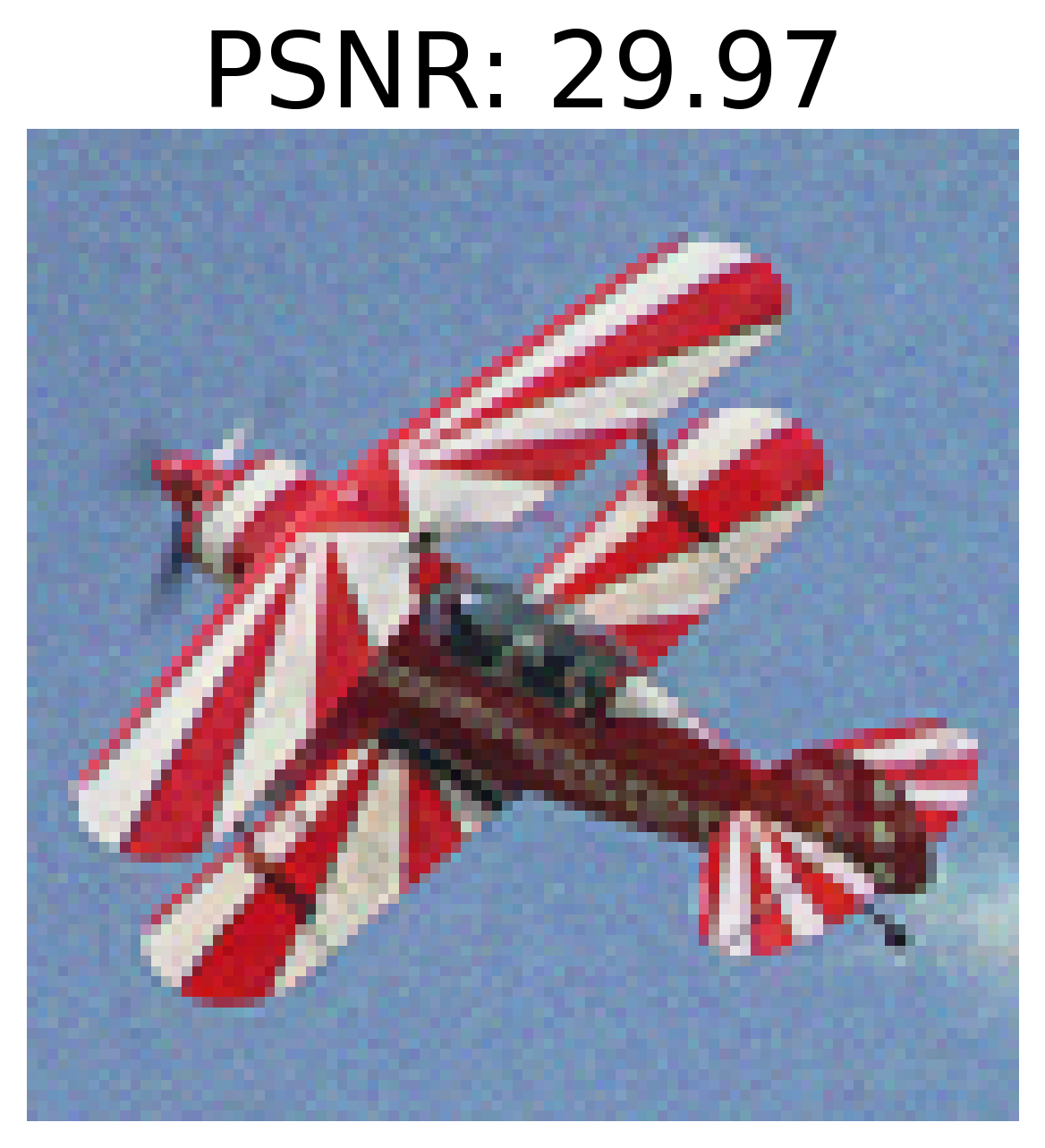}
        \caption*{\footnotesize $2 \times 10^5$}
    \end{minipage}
    \caption{\footnotesize{First row: A denoised training image after \(10^4\), \(5 \times 10^4\), \(10^5\), and \(2 \times 10^5\) computations (lower-level + linear solver iterations) using ISGD with fixed step size \(\alpha = 10^{-5}\). 
Second row: Same image with ISGD using decreasing step size \rv{(DS)}, \(\alpha_k = \frac{\alpha_0}{\sqrt{k}}\), \(\alpha_0 = 10^{-4}\).  
Third row: Same image with MAID, \(\alpha_0 = 10^{-3}\).  }}
    \label{fig:checkpoints}
    \vspace{-5pt}
\end{figure}
\subsection{Learning Field of Experts Regularizer for Deblurring}
In this section, using the FoE model as in the previous section, we consider an image deblurring problem as the lower-level problem. For training, we use $m = 1536$ color images with three channels, each of size $256 \times 256$, from the Oxford-IIIT Pet Dataset\footnote{\url{https://www.robots.ox.ac.uk/~vgg/data/pets/}}. In this case, while the upper-level problem remains the same as in \eqref{denoising_foe}, the lower-level problem is replaced with ${\hat{x}_i(\theta)} = \arg\min_{x\in \mathbb{R}^n} \frac{1}{2}\|Ax - y_i\|^2 + R_\theta(x),$ 
% \begin{equation}\label{lower_deblur}
%     {\hat{x}_i(\theta)} = \arg\min_{x\in \mathbb{R}^n} \frac{1}{2}\|Ax - y_i\|^2 + R_\theta(x), \quad i = 1,\dots, m,
% \end{equation}
where $A$  is the convolution operator corresponding to a Gaussian blur kernel of size $7\times7$ with $\sigma = 2$. Each $y_i$ is obtained by applying $A$ to $x_i^*$ and adding Gaussian noise with a noise level of $5$. 
\begin{figure}[htbp]
    \centering
    \includegraphics[width=0.5\linewidth]{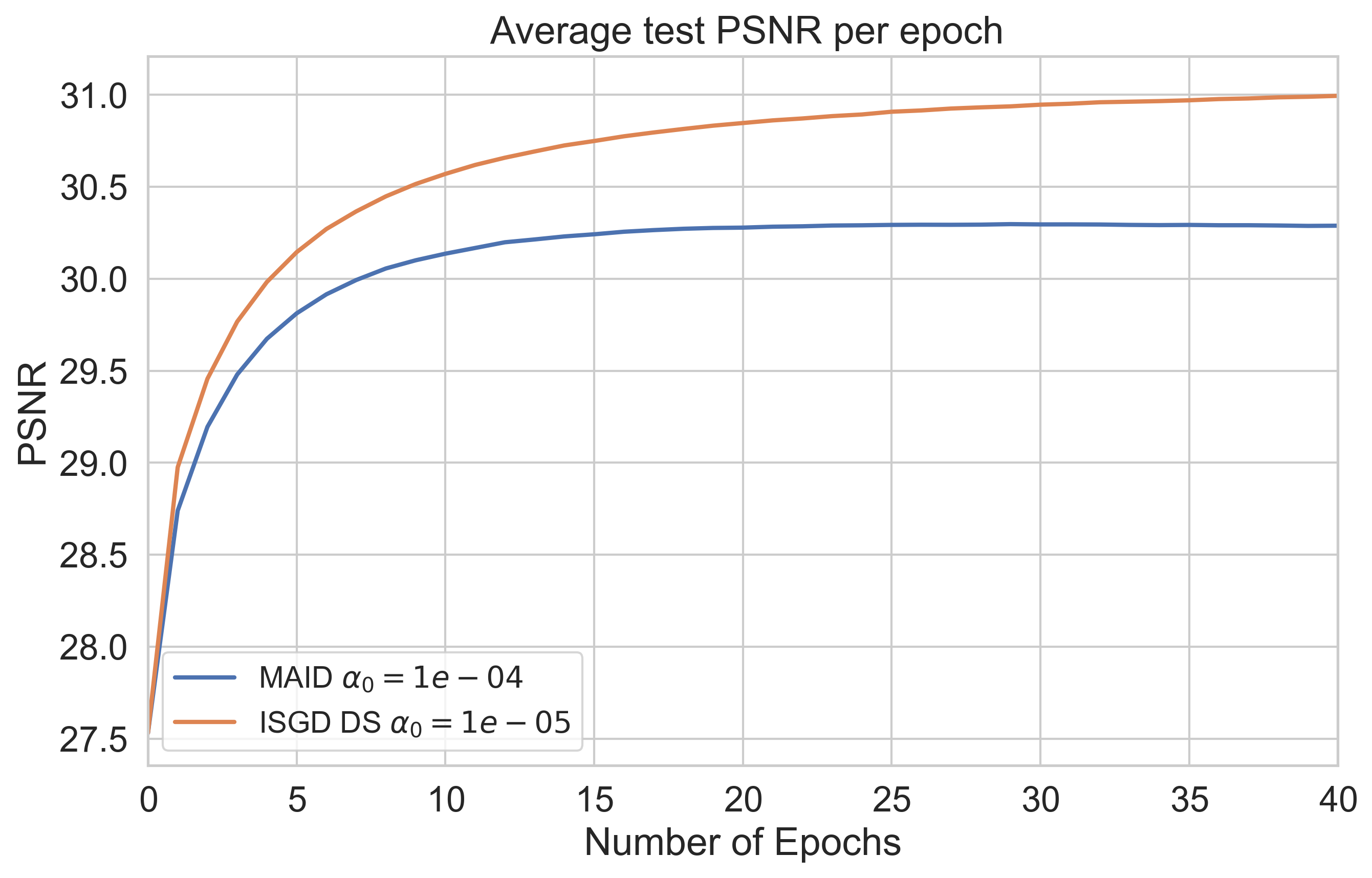}
    \caption{\footnotesize{Comparison on test dataset between the average PSNR per MAID iteration (full-batch) and average PSNR per epoch of ISGD with decreasing step size $\alpha_k = \frac{\alpha_0}{\sqrt{k}}$.}}
    \label{fig:test_psnr}
\end{figure}
\begin{figure}[htbp]\label{fig:comparison_deblur}
\centering
\begin{subfigure}{0.24\textwidth}
    \centering
    \includegraphics[width=\textwidth]{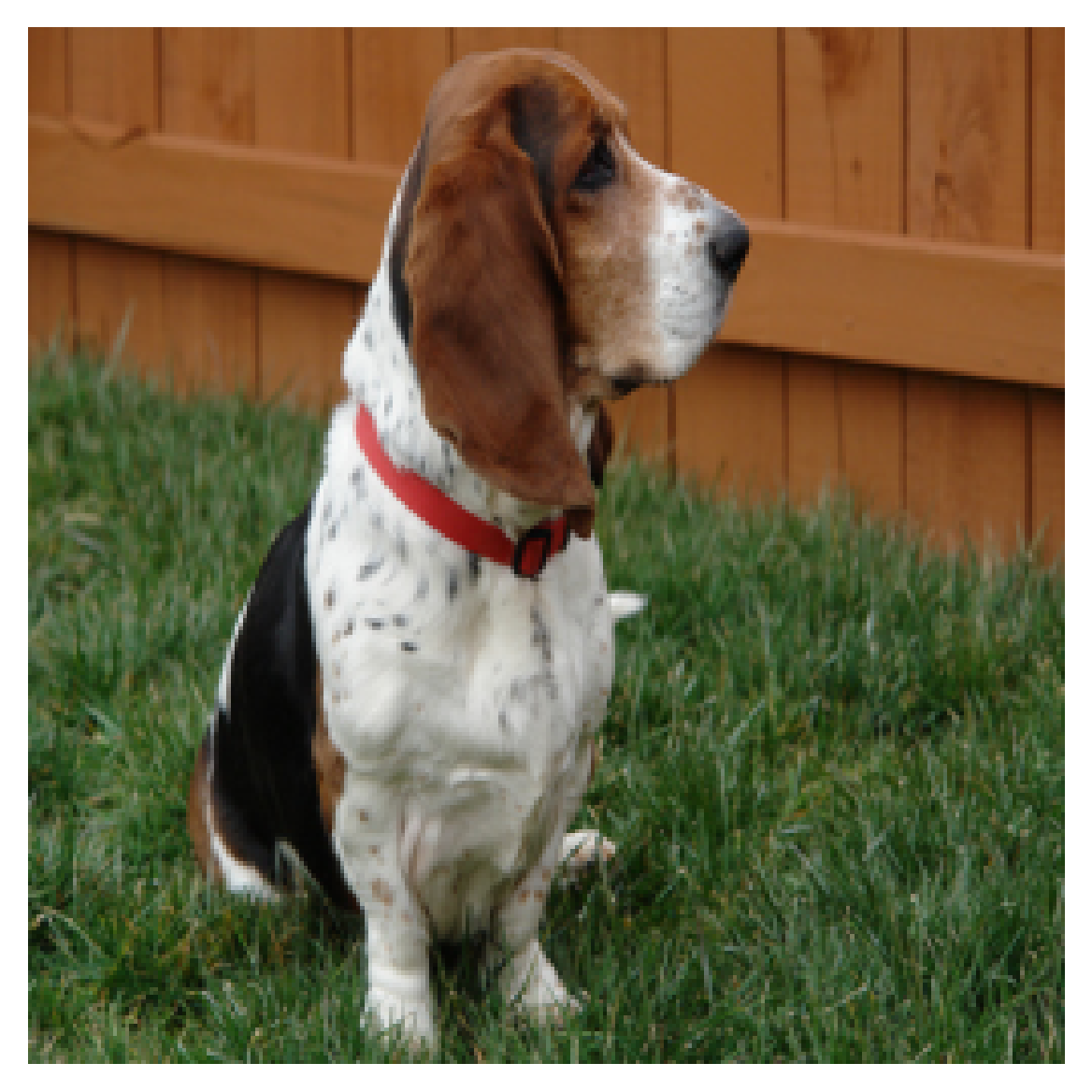}
    \caption{Ground truth}
    \label{fig:gt}
\end{subfigure}%
\hfill
\begin{subfigure}{0.24\textwidth}
    \centering
    \includegraphics[width=\textwidth]{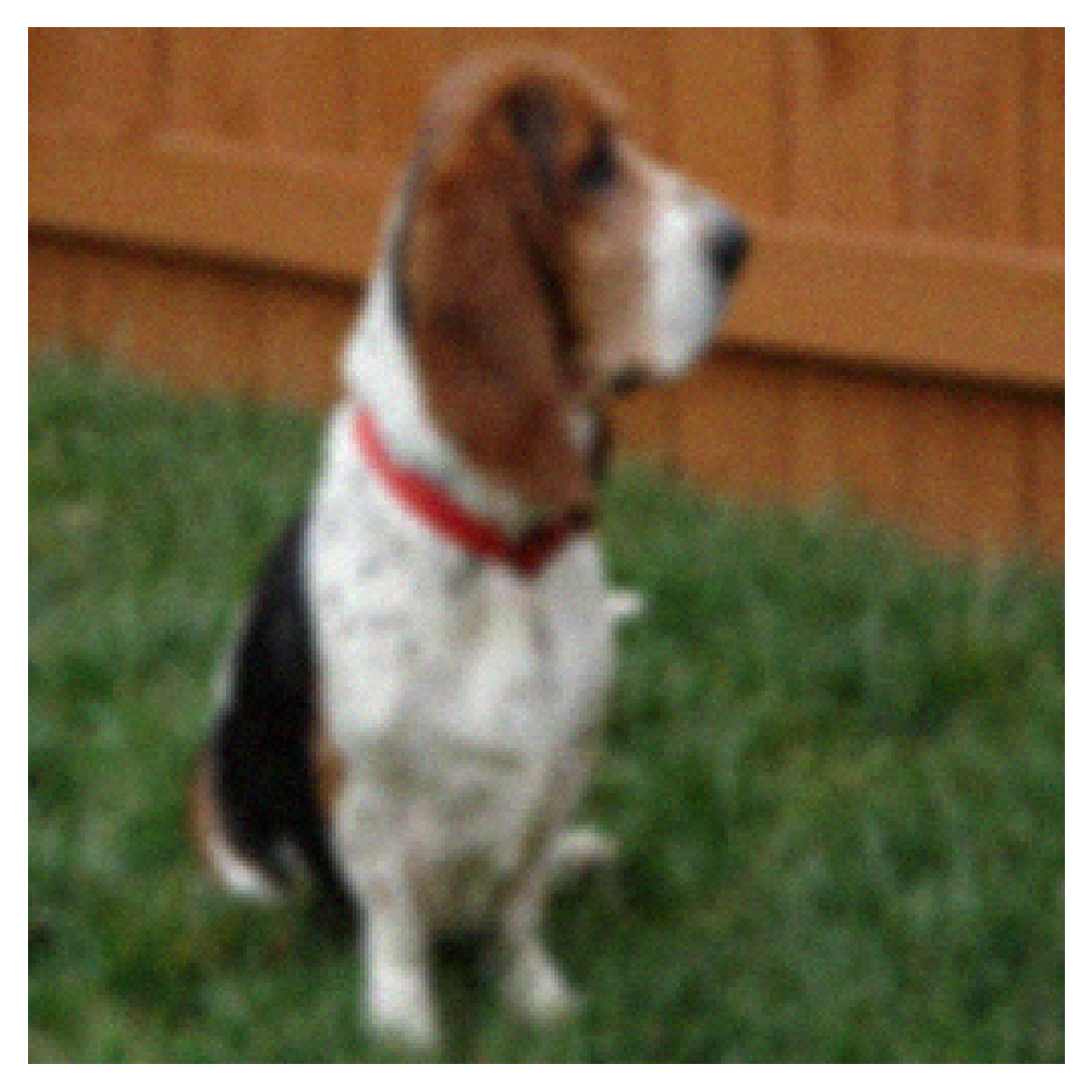}
    \caption{{Blurry: 24.78 dB}}
    \label{fig:blur}
\end{subfigure}%
\hfill
\begin{subfigure}{0.24\textwidth}
    \centering
    \includegraphics[width=\textwidth]{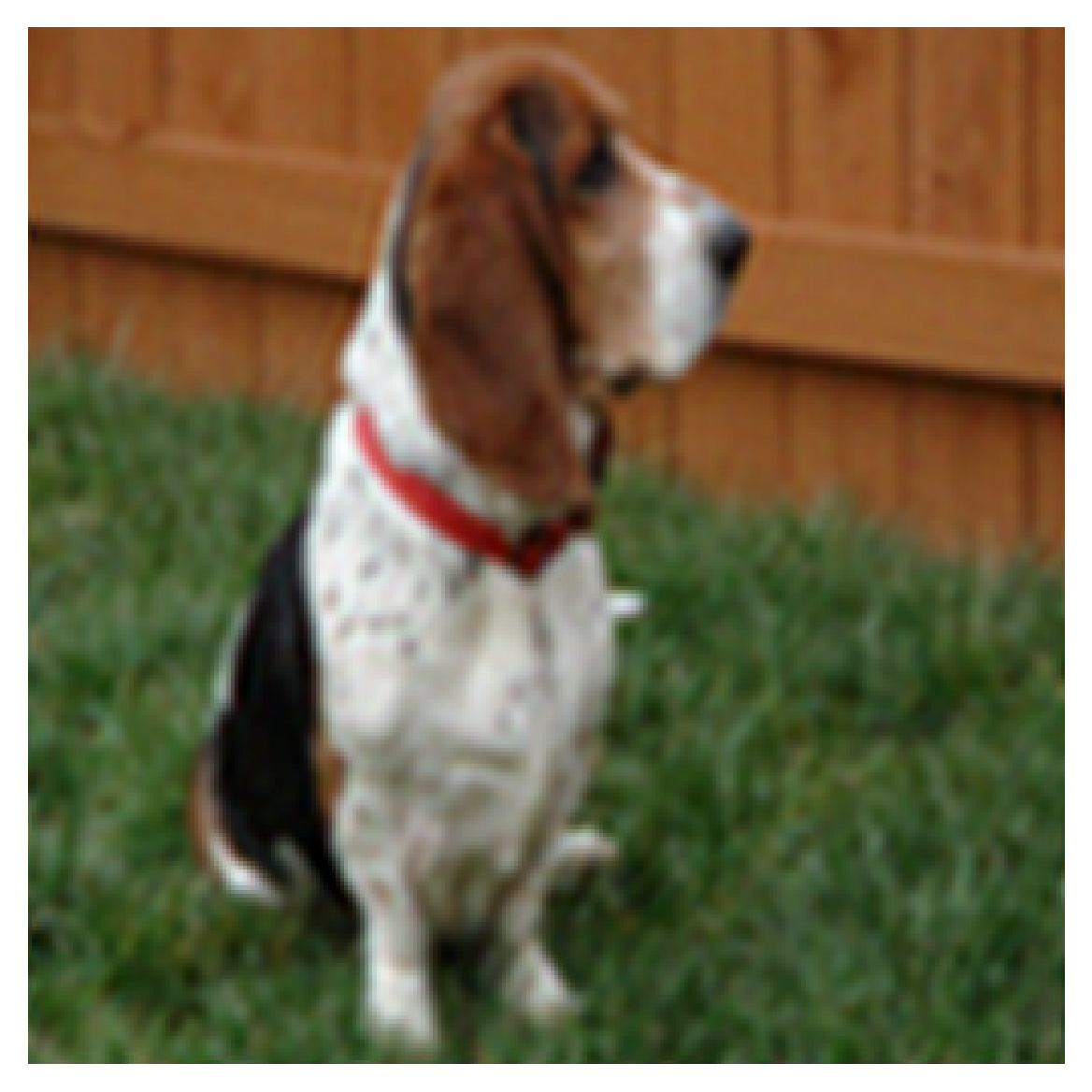}
    \caption{MAID: 27.06 dB}
    \label{fig:deblur_MAID}
\end{subfigure}%
\hfill
\begin{subfigure}{0.24\textwidth}
    \centering
    \includegraphics[width=\textwidth]{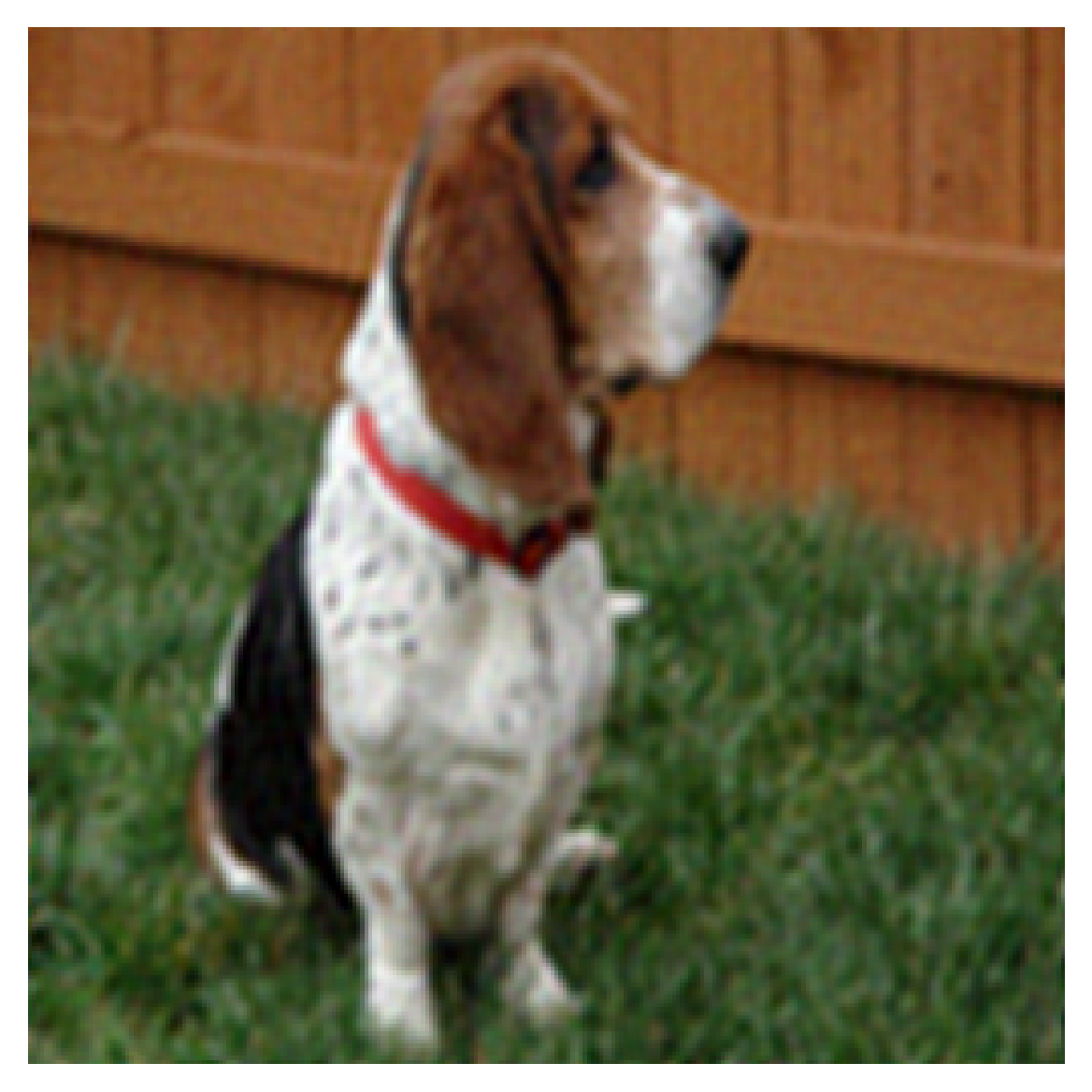}
    \caption{ISGD: 27.86 db}
    \label{fig:deblur_ISGD}
\end{subfigure}
\caption{\footnotesize{Comparison of deblurred images using the learned FoE regularizer, learned by MAID and ISGD, on test images.}}
\label{fig:comparison_deblur}

\end{figure}
In this experiment, for ISGD, we use all $m=1536$ images in mini-batches of size $S=32$. By contrast, MAID requires access to the full training dataset, which in this case exceeds the available GPU memory ($8$ GB). To allow MAID's entire dataset to be stored in memory, we only use the first m=256 images from the full dataset for MAID. Hence in this setting the data sub-sampling in ISGD allows it to use a much larger training dataset than MAID.

As shown in \cref{fig:test_psnr}, ISGD with $\alpha_0 = 10^{-4}$ and decreasing step size achieved a higher average PSNR on 32 test images after each epoch of training compared to MAID with $\alpha_0 = 10^{-3}$. This demonstrates ISGD’s ability to handle larger datasets, leading to better generalization. A quantitative comparison of the deblurred test image using the learned regularizer from MAID and ISGD is provided in \cref{fig:comparison_deblur}, where ISGD achieved a higher PSNR.
\section{Conclusions and future work}
In this work, we exploited the connection between  inexact stochastic hypergradients in bilevel problems with upper-level stochasticity on data and the practical assumptions of biased SGD. We proposed the first convergence proof for this framework under practical assumptions. Our numerical results showcase the superiority of this approach in terms of convergence speed compared to state-of-the-art deterministic bilevel learning algorithms. Moreover, it leads to a better  generalization, driven by its scalability to larger datasets. 

In the future, as seen in \cref{sec:numeric}, it would be worthy to extend our analysis to include decaying and adaptive step sizes, as well as exploring adaptive line search methods to determine suitable step sizes automatically, potentially enhancing both performance and robustness to the initial choice of the step size. 
\section*{Acknowledgments}
The work of Mohammad Sadegh Salehi was supported by a scholarship from the EPSRC Centre for
Doctoral Training in Statistical Applied Mathematics at Bath (SAMBa), under the project EP/S022945/1. Matthias J. Ehrhardt acknowledges support from the EPSRC (EP/S026045/1, EP/T026693/1, EP/V026259/1). Lindon Roberts is supported by the Australian Research Council Discovery Early Career Award DE240100006.
%
%
%
%
% ---- Bibliography ----
%
% BibTeX users should specify bibliography style 'splncs04'.
% References will then be sorted and formatted in the correct style.
%
\bibliographystyle{splncs04}
\bibliography{mybibliography}

\begin{thebibliography}{10}
\providecommand{\url}[1]{\texttt{#1}}
\providecommand{\urlprefix}{URL }
\providecommand{\doi}[1]{https://doi.org/#1}

\bibitem{piggyback_convergence}
Bogensperger, L., Chambolle, A., Pock, T.: Convergence of a piggyback-style method for the differentiation of solutions of standard saddle-point problems. SIAM Journal on Mathematics of Data Science  \textbf{4}(3),  1003--1030 (2022)

\bibitem{chambolle_pock_2016}
Chambolle, A., Pock, T.: An introduction to continuous optimization for imaging. Acta Numerica  \textbf{25},  161--319 (2016)

\bibitem{Yunjin_Chen_2014}
Chen, Y., Ranftl, R., Pock, T.: Insights into analysis operator learning: From patch-based sparse models to higher order {MRFs}. {IEEE} Transactions on Image Processing  \textbf{23}(3),  1060--1072 (mar 2014)

\bibitem{FoE}
Chen, Y., Ranftl, R., Pock, T.: Insights into analysis operator learning: From patch-based sparse models to higher order {MRFs}. {IEEE} Transactions on Image Processing  \textbf{23}(3),  1060--1072 (mar 2014)

\bibitem{Crockett_2022}
Crockett, C., Fessler, J.A.: Bilevel methods for image reconstruction. Foundations and Trends{\textregistered} in Signal Processing  \textbf{15}(2-3),  121--289 (2022)

\bibitem{biasedSGD}
Demidovich, Y., Malinovsky, G., Sokolov, I., Richt{\'a}rik, P.: A guide through the zoo of biased {SGD}. In: Thirty-seventh Conference on Neural Information Processing Systems (2023)

\bibitem{dempe_bilevel_2020}
Dempe, S., Zemkoho, A. (eds.): Bilevel Optimization: Advances and Next Challenges, vol.~161. Springer International Publishing, Cham (2020)

\bibitem{Downing_2024}
Downing, S., Gazzola, S., Graham, I.G., Spence, E.A.: Optimising seismic imaging design parameters via bilevel learning. Inverse Problems  \textbf{40}(11),  115008 (oct 2024)

\bibitem{ehrhardt2024optimalregularizationparametersbilevel}
Ehrhardt, M.J., Gazzola, S., Scott, S.J.: On optimal regularization parameters via bilevel learning. In: Data-driven Models in Inverse Problems. Radon Series on Computational and Applied Mathematics - RICAM, De Gruyter (2024)

\bibitem{boundMatthias}
Ehrhardt, M.J., Roberts, L.: {Analyzing inexact hypergradients for bilevel learning}. IMA Journal of Applied Mathematics  \textbf{89}(1),  254--278 (11 2023)

\bibitem{Meta_learning-franceschi18a}
Franceschi, L., Frasconi, P., Salzo, S., Grazzi, R., Pontil, M.: Bilevel programming for hyperparameter optimization and meta-learning. In: Proc. 35th ICML. vol.~80, pp. 1568--1577 (2018)

\bibitem{ghadimi2018approximation}
Ghadimi, S., Wang, M.: Approximation methods for bilevel programming (2018)

\bibitem{goujon2022neuralnetworkbased}
Goujon, A., Neumayer, S., Bohra, P., Ducotterd, S., Unser, M.A.: A neural-network-based convex regularizer for inverse problems. IEEE Transactions on Computational Imaging  \textbf{9},  781--795 (2022)

\bibitem{gower2019sgd}
Gower, R.M., Loizou, N., Qian, X., Sailanbayev, A., Shulgin, E., Richt{\'a}rik, P.: {SGD}: General analysis and improved rates. In: Proc. 36th ICML. vol.~97, pp. 5200--5209 (2019)

\bibitem{grazzi20a}
Grazzi, R., Franceschi, L., Pontil, M., Salzo, S.: On the iteration complexity of hypergradient computation. In: Proc. 37th ICML. vol.~119, pp. 3748--3758 (2020)

\bibitem{grazzi2021convergence}
Grazzi, R., Pontil, M., Salzo, S.: Convergence properties of stochastic hypergradients. In: Proc. 24th AISTATS. vol.~130, pp. 3826--3834 (2021)

\bibitem{Holler_2018}
Holler, G., Kunisch, K., Barnard, R.C.: A bilevel approach for parameter learning in inverse problems. Inverse Problems  \textbf{34}(11),  115012 (Sep 2018)

\bibitem{bilevel_complexity_warmstart}
Ji, K., Yang, J., Liang, Y.: Bilevel optimization: Convergence analysis and enhanced design. In: Proc. 38th ICML. vol.~139, pp. 4882--4892 (2021)

\bibitem{SGD_better}
Khaled, A., Richt{\'a}rik, P.: Better theory for {SGD} in the nonconvex world. Trans. Mach. Learn. Res.  (2023)

\bibitem{Kunisch2013ABO}
Kunisch, K., Pock, T.: A bilevel optimization approach for parameter learning in variational models. SIAM J. Imaging Sci.  \textbf{6},  938--983 (2013)

\bibitem{fullyFirstOrderStochastic}
Kwon, J., Kwon, D., Wright, S., Nowak, R.D.: A fully first-order method for stochastic bilevel optimization. In: Proc. 40th ICML. vol.~202, pp. 18083--18113 (2023)

\bibitem{bome}
Liu, B., Ye, M., Wright, S., Stone, P., qiang liu: {BOME}! bilevel optimization made easy: A simple first-order approach. In: Oh, A.H., Agarwal, A., Belgrave, D., Cho, K. (eds.) Advances in Neural Information Processing Systems (2022)

\bibitem{AD}
Mehmood, S., Ochs, P.: Automatic differentiation of some first-order methods in parametric optimization. In: Proc. 23rd AISTATS. vol.~108, pp. 1584--1594 (2020)

\bibitem{InputConvex}
Mukherjee, S., Dittmer, S., Shumaylov, Z., Lunz, S., Öktem, O., Schönlieb, C.B.: Data-driven convex regularizers for inverse problems. In: ICASSP 2024 - 2024 IEEE International Conference on Acoustics, Speech and Signal Processing (ICASSP). pp. 13386--13390 (2024)

\bibitem{Ochs}
Ochs, P., Ranftl, R., Brox, T., Pock, T.: Techniques for gradient-based bilevel optimization with non-smooth lower level problems. Journal of Mathematical Imaging and Vision  \textbf{56}(2),  175--194 (2016)

\bibitem{Pedregosa1602}
Pedregosa, F.: Hyperparameter optimization with approximate gradient (2016)

\bibitem{ramzi2023shine}
Ramzi, Z., Mannel, F., Bai, S., Starck, J.L., Ciuciu, P., Moreau, T.: {SHINE}: {SH}aring the {IN}verse estimate from the forward pass for bi-level optimization and implicit models. In: International Conference on Learning Representations (2022)

\bibitem{Reyes2023}
Reyes, J.C.D.l., Villac{\'i}s, D.: Bilevel Optimization Methods in Imaging, pp. 909--941. Springer International Publishing, Cham (2023)

\bibitem{salehi2024adaptively}
Salehi, M.S., Mukherjee, S., Roberts, L., Ehrhardt, M.J.: An adaptively inexact first-order method for bilevel optimization with application to hyperparameter learning (2024)

\bibitem{sherry2020learning}
Sherry, F., Benning, M., De~los Reyes, J.C., Graves, M.J., Maierhofer, G., Williams, G., Schönlieb, C.B., Ehrhardt, M.J.: Learning the sampling pattern for mri. IEEE Transactions on Medical Imaging  \textbf{39}(12),  4310--4321 (2020)

\bibitem{suonperä2024generalsingleloopmethodsbilevel}
Suonperä, E., Valkonen, T.: General single-loop methods for bilevel parameter learning (2024)

\end{thebibliography}
%
% \begin{thebibliography}{8}
% \bibitem{ref_article1}
% Author, F.: Article title. Journal \textbf{2}(5), 99--110 (2016)

% \bibitem{ref_lncs1}
% Author, F., Author, S.: Title of a proceedings paper. In: Editor,
% F., Editor, S. (eds.) CONFERENCE 2016, LNCS, vol. 9999, pp. 1--13.
% Springer, Heidelberg (2016). \doi{10.10007/1234567890}

% \bibitem{ref_book1}
% Author, F., Author, S., Author, T.: Book title. 2nd edn. Publisher,
% Location (1999)

% \bibitem{ref_proc1}
% Author, A.-B.: Contribution title. In: 9th International Proceedings
% on Proceedings, pp. 1--2. Publisher, Location (2010)

% \bibitem{ref_url1}
% LNCS Homepage, \url{http://www.springer.com/lncs}, last accessed 2023/10/25
% \end{thebibliography}
\end{document}